\documentclass{amsart}

\title{Intrinsically spherical 3-linked graphs}
\author[M. Burkhart]{Madeleine Burkhart}
\author[A. Castillo]{Andrew Castillo}
\author [J. Doane]{Jonathan Doane}
\author[J. Foisy]{Joel Foisy}
\author[C. Negron]{Cristopher Negron}

\address[Madeleine Burkhart]{c/o SUNY Potsdam Mathematics Department \\ 44 Pierrepont Avenue\\Potsdam, NY 13676 \\}
\address[Andrew Castillo]{c/o SUNY Potsdam Mathematics Department \\ 44 Pierrepont Avenue\\Potsdam, NY 13676\\azcastillo316@gmail.com}
\address[Jonathan Doane]{Binghamton University Department of Mathematical Sciences\\Binghamton University \\ PO Box 6000\\Binghamton, NY 13902-6000\\ doane@math.binghamton.edu}
\address[Joel Foisy]{SUNY Potsdam Mathematics Department \\MacVicar Hall 223, 44 Pierrepont Avenue \\Potsdam, NY 13676 \\foisyjs@potsdam.edu}
\address[Cristopher Negron]{c/o SUNY Potsdam Mathematics Department \\ 44 Pierrepont Avenue\\Potsdam, NY 13676\\cristopher.negron@nn.k12.va.us}

\usepackage{amsfonts}
\usepackage{amssymb}
\usepackage{graphicx}
\usepackage{float}
\usepackage{latexsym}
\usepackage{amsthm}
\usepackage{amscd}
\usepackage{enumerate}
\usepackage{epsfig}
\usepackage{yfonts}
\usepackage{relsize}
\usepackage{appendix}
\usepackage{multicol}

\newtheorem{theorem}{Theorem}[section]

\newtheorem{proposition}[theorem]{Proposition}
\newtheorem{corollary}[theorem]{Corollary}
\newtheorem{conjecture}[theorem]{Conjecture}

\theoremstyle{definition}

\newcommand{\N}{{\mathbb N}}

\newcommand{\disunion}{\dot{\bigcup}}

\newcommand{\dU}{\dot{\bigcup}}

\newcommand{\Sph}{$S^2$}

\newcommand{\iplI}{intrinsically type I spherical 3-linked}
\newcommand{\mmI}{minor-minimal with respect to being intrinsically type I spherical 3-linked}
\newcommand{\iplII}{intrinsically type II planar 3-linked}
\newcommand{\mmII}{minor-minimal with respect to being intrinsically type II spherical 3-linked}

\newcommand{\Don}{$\mathfrak{D}_1$}
\newcommand{\Dtw}{$\mathfrak{D}_2$}
\newcommand{\Dth}{$\mathfrak{D}_3$}
\newcommand{\Dfo}{$\mathfrak{D}_4$}
\newcommand{\Dfi}{$\mathfrak{D}_5$}
\newcommand{\Dsi}{$\mathfrak{D}_6$}
\newcommand{\Dsea}{$\mathfrak{D}_{7_{\alpha}}$}
\newcommand{\Deia}{$\mathfrak{D}_{8_{\alpha}}$}
\newcommand{\Dnia}{$\mathfrak{D}_{9_{\alpha}}$}
\newcommand{\Dtena}{$\mathfrak{D}_{10_{\alpha}}$}
\newcommand{\Delea}{$\mathfrak{D}_{11_{\alpha}}$}
\newcommand{\Dtwe}{$\mathfrak{D}_{12}$}
\newcommand{\Dthi}{$\mathfrak{D}_{13}$}
\newcommand{\Dons}{$\mathfrak{D}_1$ ~}
\newcommand{\Dtws}{$\mathfrak{D}_2$ ~}
\newcommand{\Dths}{$\mathfrak{D}_3$ ~}
\newcommand{\Dfos}{$\mathfrak{D}_4$ ~}
\newcommand{\Dfis}{$\mathfrak{D}_5$ ~}
\newcommand{\Dsis}{$\mathfrak{D}_6$ ~}
\newcommand{\Dseas}{$\mathfrak{D}_{7_{\alpha}}$\space}
\newcommand{\Deias}{$\mathfrak{D}_{8_{\alpha}}$\space}
\newcommand{\Dnias}{$\mathfrak{D}_{9_{\alpha}}$\space}
\newcommand{\Dtenas}{$\mathfrak{D}_{10_{\alpha}}$\space}
\newcommand{\Deleas}{$\mathfrak{D}_{11_{\alpha}}$\space}
\newcommand{\Dtwes}{$\mathfrak{D}_{12}$ ~}
\newcommand{\Dthis}{$\mathfrak{D}_{13}$ ~}
\newcommand{\Dseb}{$\mathfrak{D}_{7_{\beta}}$}
\newcommand{\Deib}{$\mathfrak{D}_{8_{\beta}}$}
\newcommand{\Dnib}{$\mathfrak{D}_{9_{\beta}}$}
\newcommand{\Dtenb}{$\mathfrak{D}_{10_{\beta}}$}
\newcommand{\Deleb}{$\mathfrak{D}_{11_{\beta}}$}
\newcommand{\Dsebs}{$\mathfrak{D}_{7_{\beta}}$ \space}
\newcommand{\Deibs}{$\mathfrak{D}_{8_{\beta}}$ \space}
\newcommand{\Dnibs}{$\mathfrak{D}_{9_{\beta}}$ \space}
\newcommand{\Dtenbs}{$\mathfrak{D}_{10_{\beta}}$ \space}
\newcommand{\Delebs}{$\mathfrak{D}_{11_{\beta}}$ \space}

\newcommand{\demb}{distinct embeddings into \Sph, up to equivalence}
\newcommand{\equi}{equivalence classes of edges}
\newcommand{\gosd}{is the graph obtained by applying a sub-dangle move to ~}
\newcommand{\govb}{is the graph obtained by applying a vert-bar move to ~}
\newcommand{\srsd}{satisfies the requirements to apply the sub-dangle move}

\newcommand{\iplIIns}{intrinsically type II spherical 3-linked}

\newcommand{\VBg}{$G_0 \dU K_2$ }
\newcommand{\VBgme}{$G_0 \dU K_2$ \textit{minus an edge} }
\newcommand{\VBgmv}{$G_0 \dU K_2$ \textit{minus a vertex} }
\newcommand{\VBgce}{$G_0 \dU K_2$ \textit{with a contracted edge} }

\begin{document}
\setcounter{section}{0}

\maketitle

\noindent \underline{Abstract:} 
\space   
We exhibit several families of planar graphs that are minor-minimal intrinsically spherical $3$-linked. A graph is \textit {intrinsically spherical 3-linked} if it is planar graph that has, in every
spherical embedding, a non-split 3-link consisting of two disjoint cycles ($S^1$s) and
two disjoint vertices ($S^0$), or a cycle and two pairs of disjoint vertices. We conjecture that $K_4 \dU K_4$, $K_{3,2} \dU K_{3,2}$, and $K_4 \dU K_{3,2}$ form the complete set of minor-minimal intrinsically \textit{type I spherical 3-linked} graphs (that is, in every spherical embedding, have a nonsplit link of two cycles and one $S^0$).

%
%
%
%
%
%
%
%


\section{Introduction} \label{sec:intro}

Dekhordi and Farr \cite{DF} showed that the complete set of minor-minimal intrinsically spherical linked graphs consists of $K_{4}\disunion K_1$, $K_{3,2}\disunion K_1$, and $K_{3,1,1}$, where $K_n$ stands for the complete graph on $n$ vertices, $\disunion$ stands for the disjoint union, and $K_{n_1,n_2,...,n_k}$ stands for the complete multipartite graph on $n_1+...n_k$ vertices, with $k$ partition sets. \textit {Intrinsically spherical linked graphs} (also known as \textit {separating planar graphs}) are planar graphs that have every
spherical embedding containing a non-split link consisting of a cycle ($S^1$) and
two disjoint vertices ($S^0$). Here, we exhibit several families of planar graphs that are minor-minimal \textit {intrinsically spherical $3$-linked}, that is planar graphs that have, in every
spherical embedding, a non-split 3-link consisting of two disjoint cycles ($S^1$s) and
two disjoint vertices ($S^0$), or a cycle and two pairs of disjoint vertices.

During the 1980's, Conway and Gordon \cite{1}, and Sachs \cite{3}, \cite{5}, independently proved that $K_6$ is intrinsically linked, that is, every spatial embedding of $K_6$ contains a pair of cycles that forma a non-split link.  Robertson, Seymour, and Thomas \cite{RST} proved Sachs' conjecture, that the Petersen family of graphs (the seven graphs including $K_6$, obtained from $K_6$ by $\Delta-Y$ and $Y-\Delta$ exchanges) form the complete minor-minimal set of graphs that are intrinsically linked.


Robertson and Seymour's Minors Theorem \cite{RS} states that if $P$ is a minor-closed graph property, then the minor-minimal forbidden graphs for $P$ form a finite set. This implies that the complete set of minor-minimal intrinsically 3-linked graphs must be finite. Recall that a graph is \textit{intrinsically 3-linked} if it contains, in every spatial embedding, cycles that form a non-split link of 3 components. It is still open to classify the complete set of minor-minimal intrinsically 3-linked graphs, and it appears that the list of such graphs will be significantly larger than 7 (see, for example, \cite {FFNP}, \cite{BF}, \cite{OD}). The difficulty of this problem has been the inspiration and motivation for this paper; we aimed to find the complete set of minor-minimal intrinsically spherical 3-linked graphs.  The hope is that this question will be easier for the spherical case. Though the question remains open, we conjecture that $K_4 \dU K_4$, $K_{3,2} \dU K_{3,2}$, and $K_4 \dU K_{3,2}$ form the complete set of minor-minimal intrinsically \textit{type I spherical 3-linked} graphs (that is, in every spherical embedding, have a nonsplit link of two cycles and one $S^0$). We also exhibit several graphs that are minor-minimal intrinsically \textit{type II spherical 3-linked}. In general, it remains open to find families of planar graphs that are minor-minimal intrinsically (spherical) $n-$linked for $n \geq 3$.

\section{Terminology and Background}

First, let us define $S^k$ to be the $k$-sphere.  We call a particular way to place a graph, $G$, into $S^k$ an \textit{embedding} of $G$.  A \textit{link} is a collection of disjoint spheres of various dimensions, embedded within a sphere of greater dimension.  We say that a link, $\ell$, embedded in $S^k$ is \textit{split} if there exists an $S^{k-1}$ embedded in $S^{k} - \ell$ that bounds only part of $\ell$.  We say that a graph, $G$, is \textit{intrinsically linked} if every embedding of $G$ into $S^3$ contains a \textit{non-split} link.

\begin{figure}[H]
\centering
\includegraphics[width=.9\linewidth]{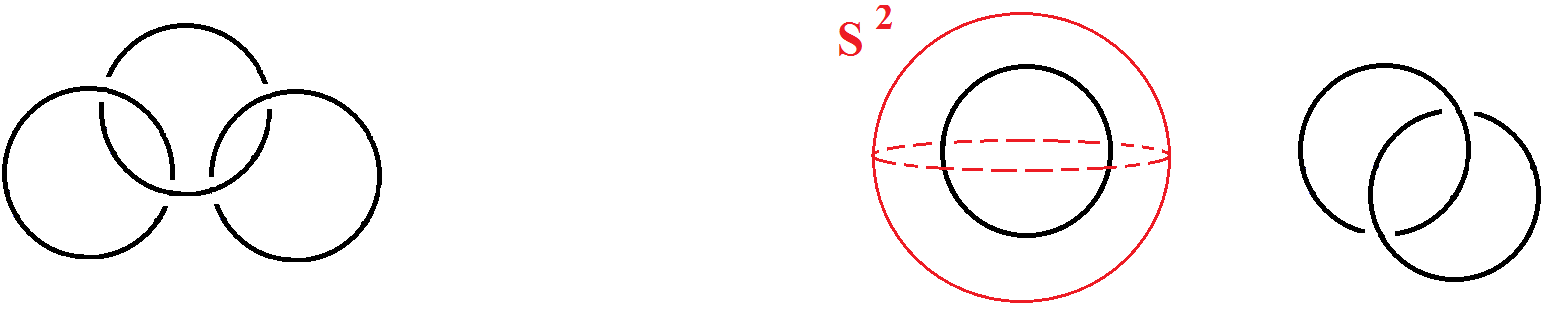}
\caption{[Left]  A \textit{non-split} link of three components.  [Right]  A \textit{split} link of three components.}
\label{fig:splitlink}
\end{figure}

A planar embedding of a graph, $G$, is \textit{spherical linked} if it contains a non-split link of one embedded $S^1$ and one embedded $S^0$, as depicted below:

\begin{figure}[H]
\centering
\includegraphics[width=.2\textwidth, height=.1\textheight]{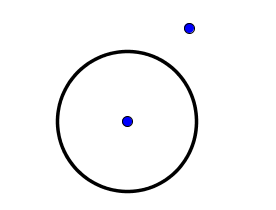}
\end{figure}

We say that a graph, $G$, is \textit{intrinsically spherical linked} if every embedding of $G$ into \Sph ~contains a non-split spherical link.

A spherical embedding of a graph, $G$, is \textit{type I spherical 3-linked} if it contains a non-split link of two embedded $S^1$'s and one embedded $S^0$, as depicted below:

\begin{figure}[H]
\centering
\includegraphics[width=0.7\textwidth]{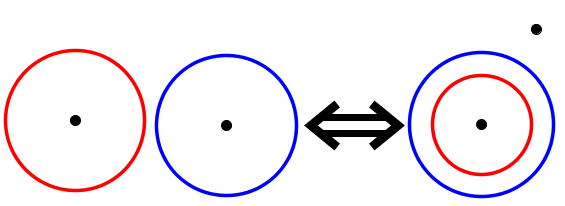}
\label{fig:Type1a}
\end{figure}

A planar embedding of a graph, $G$, is \textit{type II spherical 3-linked} if $G$ contains a non-split link of one embedded $S^1$ and two embedded $S^0$'s, as depicted below:

\begin{figure}[H]
\centering
\includegraphics[width=.3\textwidth, height=.1\textheight]{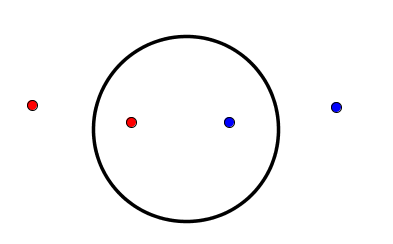}
\end{figure}


When speaking of planar links, it doesn't make topological sense to call 0-spheres ``components," since an $S^{0}$ is not connected.  Henceforth we will refer to an $S^{1}$ or an $S^{0}$ as a \textit{piece} of an $n$-link.  Note that the following definition only makes sense \textit{after} we have chosen which points form $S^{0}$s.






\begin{theorem} (Dekhordi and Farr \cite{DF}). 
A graph, $G$, is intrinsically spherical linked if and only if $G$ contains $K_4 \dU K_1$, $K_{3,2} \dU K_1$, or $K_{3,1,1}$ as a minor.
\end{theorem}




\medskip
A \textit{spherical $n$-link} is a disjoint collection of $n-m$ 1-spheres and $m$ 0-spheres, embedded into \Sph.  
We say a spherical $n$-link, $\ell$, is \textit{split} if  there exists an $S^1$ embedded in $S^{2} - \ell$ that bounds only part of $\ell$.  
We say that a spherical embedding of a graph, $G$, is \textit{linked} if it contains a non-split link of one $S^1$ and one $S^0$.

A graph, $H$, is a \textit{minor} of another graph, $G$, if and only if $H$ can be obtained from $G$ by a sequence of the following: \textit{vertex deletion(s)}, \textit{edge deletion(s)}, (and/or) \textit{edge contraction(s)}.  We say that a graph, $G$, is \textit{minor-minimal} (with respect to a property) if $G$ has that property and no minor of $G$ has that property.

A graph, $G$, is \textit{planar} if and only if $G$ can be embedded into $S^2$. Kuratowski's famous theorem \cite{2} asserts that a graph, $G$, is planar if and only if $G$ does not contain $K_5$ nor $K_{3,3}$ as a minor. 

%
%

A graph, $G$, is \textit{outerplanar} if $G$ can be embedded in $S^{2}$ with all vertices on a common face.

%
%


\begin{theorem}\cite{halin}, \cite{CH} \label{thm:nonouterplanar}
A graph, $G$, is outerplanar if and only if $G$ does not contain $K_{4}$ nor $K_{3,2}$ as a minor. 
\end{theorem}

A planar graph, $G$, is \textit{$n$-connected} if $G$ is connected, and removing $n - 1$ or fewer vertices from $G$ always results in a connected planar graph.



\section{Intrinsically Spherical 3-Linked Graphs} \label{sec:ip3l}

\medskip

\subsection{Type I Spherical 3-Linked Graphs} \space

A spherical embedding of a graph, $G$, is \textit{type I spherical 3-linked} if it contains a non-split link of two embedded $S^1$'s and one embedded $S^0$, as depicted in Figure \ref{fig:type1}:

\begin{figure}[H]
\centering
\includegraphics[width=0.7\textwidth]{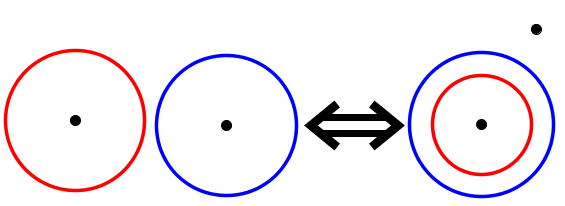}
\label{fig:type1}
\caption{A type I non-split 3-link.}
\end{figure}

\vskip -.5in

\begin{figure}[H]
\centering
\includegraphics[width=0.8\linewidth]{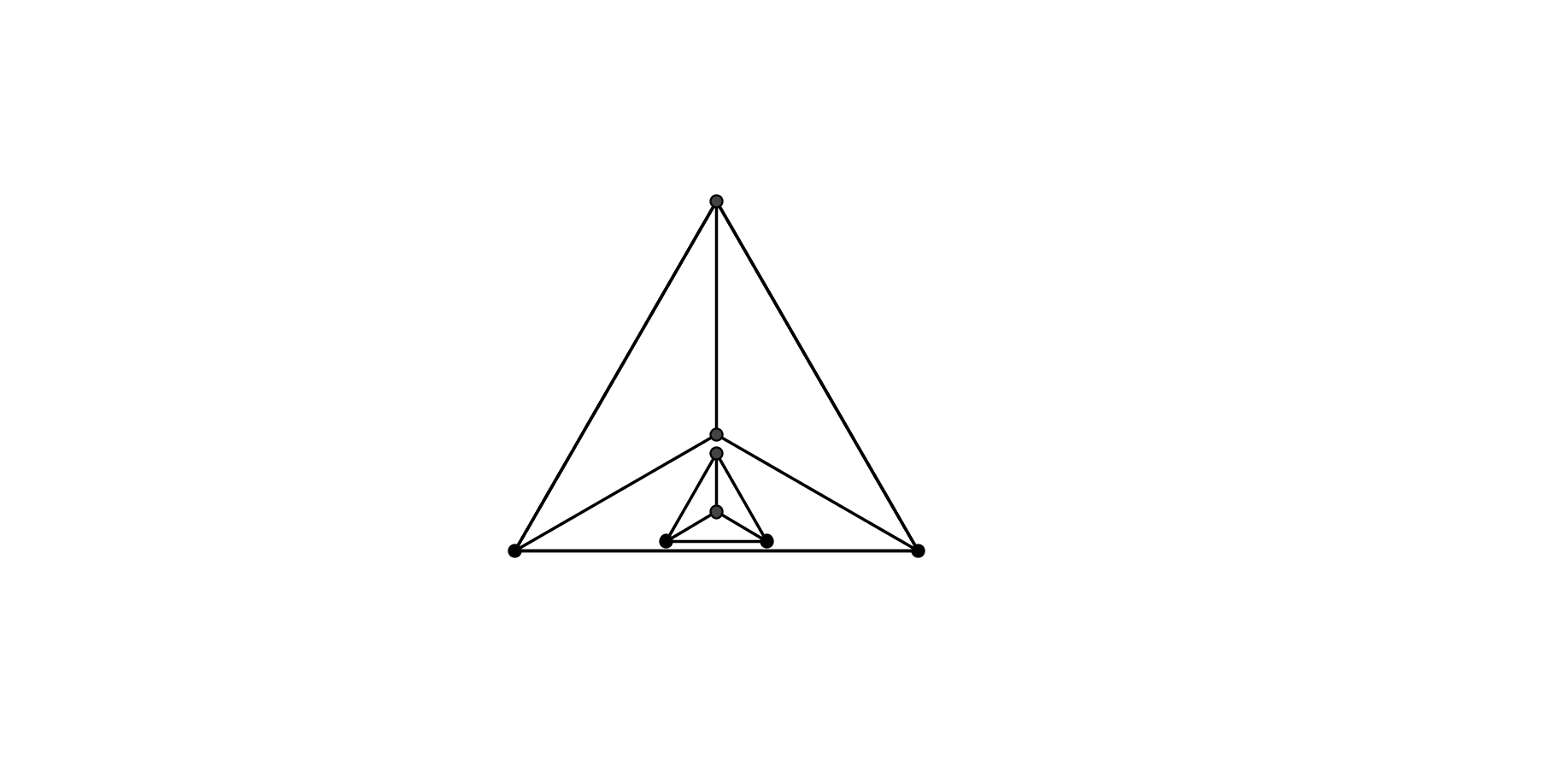}
\caption{$K_4 \dU K_4$ has a unique embedding in \Sph, up to equivalence.}
\label{fig:K4UdotK4together}
\end{figure}

\vskip -.5in

\begin{figure}[H]
\centering
\includegraphics[width=0.8\linewidth]{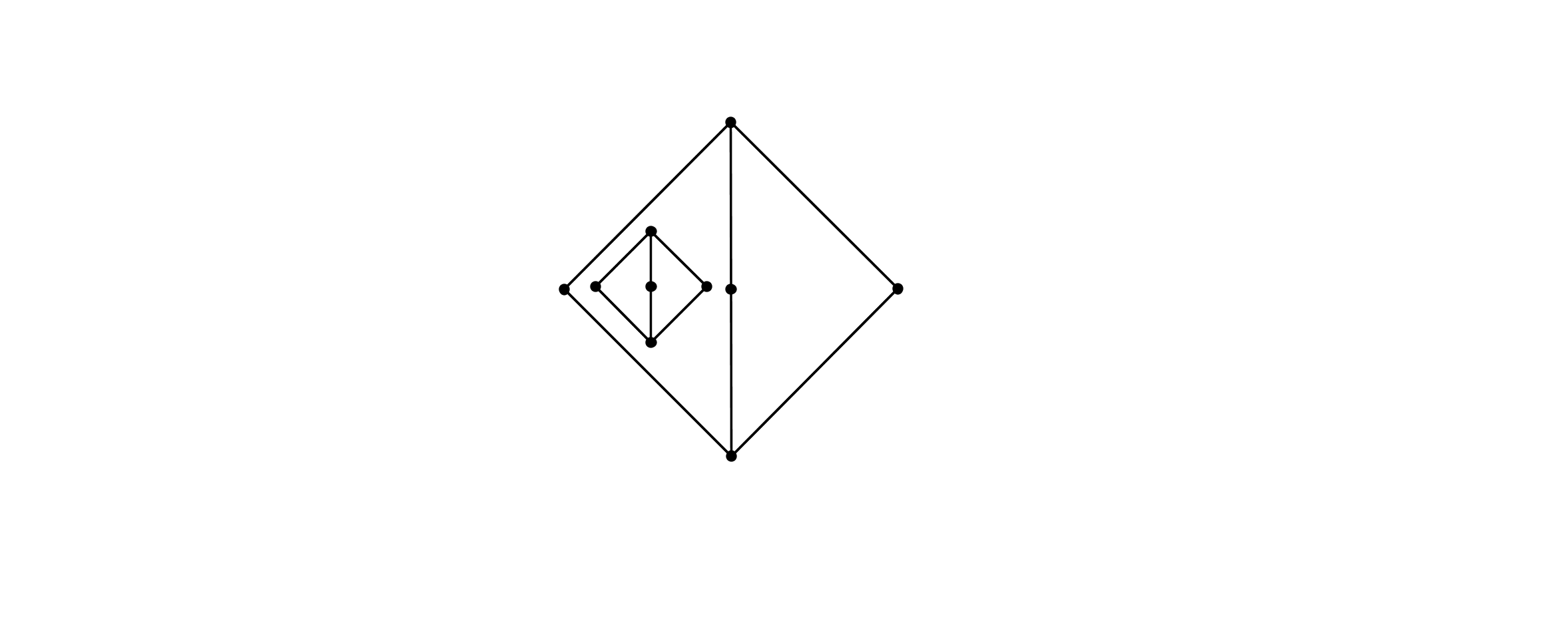}
\caption{$K_{3,2} \dU K_{3,2}$ has a unique embedding in \Sph, up to equivalence.}
\label{fig:K3,2UdotK3,2together}
\end{figure}

\vskip -.15in
\begin{figure}[H]
\centering
\includegraphics[width=0.7\linewidth]{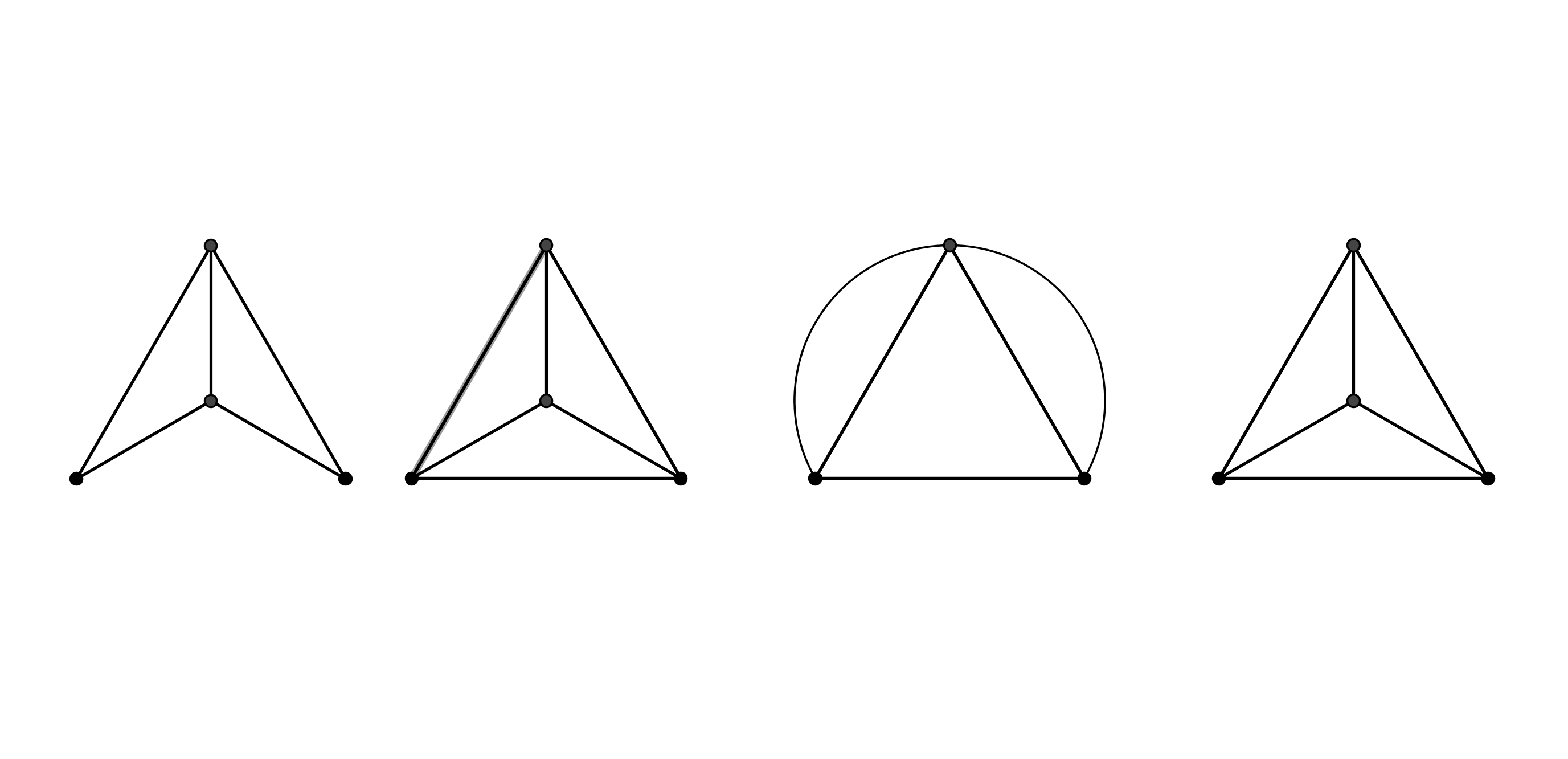}
\vskip -.15in
\caption{[Left] $K_4 \dU K_4$ minus an edge. [Right] $K_4 \dU K_4$ with a contracted edge.}
\label{fig:K4minors}
\end{figure}

\begin{figure}[H]
\centering
\includegraphics[width=0.7\linewidth]{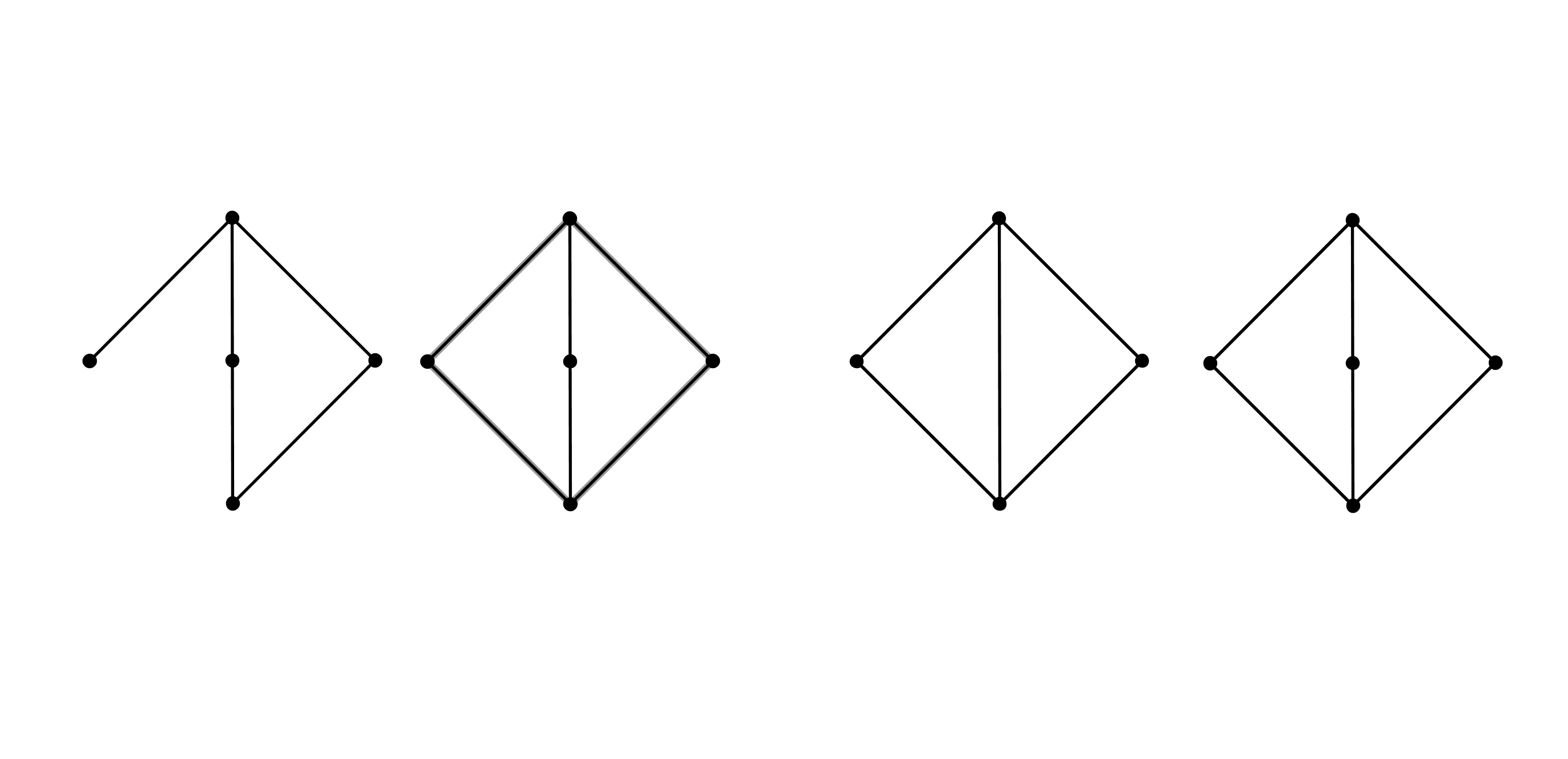}
\caption{[Left] $K_{3,2} \dU K_{3,2}$ minus an edge. [Right] $K_{3,2} \dU K_{3,2}$ with a contracted edge.}
\label{fig:K3minors}
\end{figure}

A graph, $G$, is \textit{intrinsically type I spherical 3-linked} if every spherical embedding of $G$ is type I spherical 3-linked.

\begin{proposition} \label{prop:K4K3,2}
The graphs $K_{4} \dot{\bigcup} K_{4}$, $K_{3,2} \dot{\bigcup} K_{3,2}$, and $K_{4} \dot{\bigcup} K_{3,2}$ are minor-minimal with respect to being intrinsically type I spherical 3-linked.
\end{proposition}

\begin{proof}

First, we embed $K_4$ into \Sph.  This embedding determines four faces, each of which are equivalent.  Now, we embed another $K_4$ into our \Sph.  As each face was equivalent, we have that $K_4 \dU K_4$ has a unique planar embedding, up to equivalence.  By examining Figure \ref{fig:K4UdotK4together}, we see that $K_4 \dU K_4$ is \iplI.

Now, we will verify that $K_4 \dot{\bigcup} K_4$ is \mmI.  
Before we begin, we note that each edge in $K_4 \dot{\bigcup} K_4$ is equivalent to any other edge.  
Now, by examining Figure \ref{fig:K4minors} [Left], we see that $K_4 \dot{\bigcup} K_4$ minus an edge fails to be \iplI.  Thus, if we delete any vertex from $K_4 \dot{\bigcup} K_4$ the resulting graph will not be \iplI, since each vertex is connected to an edge in $K_4 \dot{\bigcup} K_4$.  
Finally, by examining Figure \ref{fig:K4minors} [Right], we see that $K_4 \dot{\bigcup} K_4$ with a contracted edge fails to \iplI.  Hence, no minor of $K_4 \dot{\bigcup} K_4$ is \iplI.  Therefore, $K_4 \dot{\bigcup} K_4$ is \mmI.

Next, we embed $K_{3,2}$ into \Sph.  This embedding determines three faces, each of which are equivalent.  Now, we embed another $K_{3,2}$ into our \Sph.  As each face was equivalent, we have that $K_{3,2} \dU K_{3,2}$ has a unique planar embedding, up to equivalence.  By examining Figure \ref{fig:K3,2UdotK3,2together}, we see that $K_{3,2} \dU K_{3,2}$ is \iplI.

Now, we will verify that $K_{3,2} \dU K_{3,2}$ is \mmI.  
Before we begin, we note that each edge in $K_{3,2} \dU K_{3,2}$ is equivalent to any other edge.  Now, by examining Figure \ref{fig:K3minors} [Left], we see that $K_{3,2} \dU K_{3,2}$ minus an edge fails to be \iplI.  Thus, if we delete any vertex from $K_{3,2} \dU K_{3,2}$ the resulting graph will not be \iplI, since each vertex is connected to an edge in $K_{3,2} \dU K_{3,2}$.  Finally, by examining Figure \ref{fig:K3minors} [Right], we see that $K_{3,2} \dU K_{3,2}$ with a contracted edge fails to \iplI.  Hence, no minor of $K_{3,2} \dU K_{3,2}$ is \iplI.  Therefore, $K_{3,2} \dU K_{3,2}$ is \mmI.

Lastly, we note that the proof that $K_4 \dU K_{3,2}$ is \mmI ~is analogous to the above argument.

This completes our proof.
\end{proof}

\begin{conjecture}
The graphs $K_4 \dU K_4$ $K_{3,2} \dU K_{3,2}$, and $K_4 \dU K_{3,2}$ form the complete minor-minimal set of intrinsically type I spherical 3-linked graphs.
\end{conjecture} 
 

\subsection{Type II Spherical 3-Linked} \space

A planar embedding of a graph, $G$, is \textit{type II spherical 3-linked} if $G$ contains a non-split link of one embedded $S^1$ and two embedded $S^0$'s, as depicted in Figure \ref{II}.

\begin{figure}[H]
\centering
\includegraphics[width=.3\textwidth, height=.1\textheight]{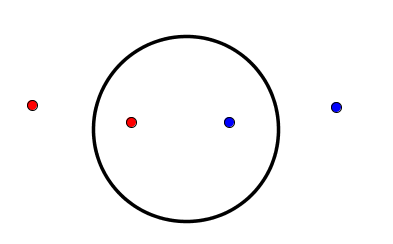}
\caption{A type II non-split spherical 3-link.}
\label{II}
\end{figure}

We say that a graph, $G$, is \textit{intrinsically type II spherical 3-linked} if every spherical embedding of $G$ is type II spherical 3-linked.

\subsubsection{Vertices-Bar Exchange} \space
In this sub-section, we will define the Vertex-Bar exchange, and then show how it preserves the property of being intrinsically minor-minimal intrinsically type II spherical linked.

\medskip

\begin{figure}[H]
\centering
\includegraphics[width=0.7\linewidth]{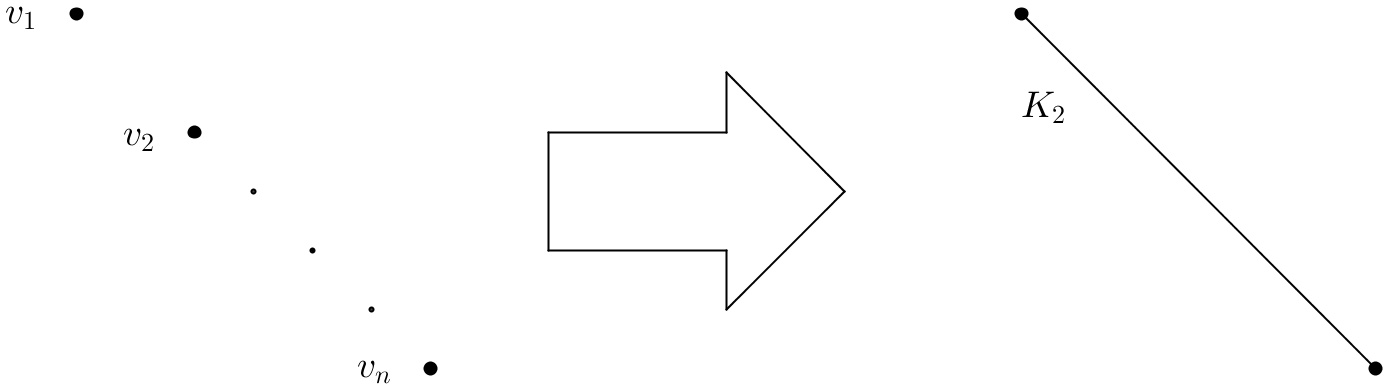}
\caption{An illustration of the Vertices-Bar exchange.}
\label{fig:VertexBarMove}
\end{figure}

\begin{proposition}
Suppose that a graph, $G$, satisfies the following:  \begin{enumerate}[i]

\item  $G$ is \mmII,  

\item  $G = G_0 \dU \{v_1, v_2, ... , v_n\}$, for some $n \in \N$ with $n \geq 3$ (where $G_0$ is a connected planar graph),

\item  $G_0 \dU \{v_1, v_2\}$ contains a type II spherical 3-link whenever $G_0$ is embedded into \Sph and the two vertices $v_1$ and $v_2$ are embedded into any one face of $G_0$,

\item   In each embedding $\varphi$ of $G_0 - e$ into \Sph, there is at least one face, $F_\varphi$, such that $G_0 \dU \{v_1, v_2, ... , v_n\}$ is not type II spherical 3-linked when $\{v_1, v_2, ... , v_n\}$ is embedded into $F_\varphi$, 

\item   In each embedding $\varphi$ of $G_0 \backslash e$ into \Sph, there is at least one face, $F_\phi$, such that $G_0 \dU \{v_1, v_2, ... , v_n\}$ is not type II spherical 3-linked when $\{v_1, v_2, ... , v_n\}$ is embedded into $F_\phi$. 
\end{enumerate}
\noindent Then, \VBg is \mmII.  

\end{proposition}

\noindent
\textit{Proof.}  Suppose that $G$ is a graph that satisfies the properties above.  

Notice, $K_2$  contains exactly two vertices and can only possibly be embedded into any one face of each embedding of $G_0$ into \Sph.  Thus, by hypothesis $iii$, \VBg is \iplIIns.  

Now, we will verify that \VBg is \mmII.  
Note that each edge in $G$ is exclusively contained in $G_0$ or $K_2$. 
Consider $G_0 \dU K_2 ~minus~an~edge$.  By hypothesis $iv$, we know that there exists an embedding $\varphi$ of \VBgme that is not type II spherical 3-linked whenever we delete an edge contained in $G_0$; namely, when we embed $K_2$ into $F_\varphi$.  
Furthermore, since \VBgme equals $G_0 \dU \{v_1, v_2\}$ when we delete the edge contained in $K_2$, we know that each embedding of \VBgme fails to be \iplIIns, by hypotheses $i$ and $ii$.  

Hence, as $G_0$ is connected, we have that each embedding of \VBgmv is not \iplIIns, as well, since we cannot delete a vertex from \VBg without also deleting an edge.  

Lastly, consider $G_0 \dU K_2 ~with~a~contracted~edge$.  
By hypothesis $v$, we know that there exists an embedding $\phi$ of \VBgce that is not type II spherical 3-linked whenever we contract an edge contained in $G_0$; namely, when we embed $K_2$ into $F_\phi$.  Furthermore, since \VBgce equals $G_0 \dU \{v_1\}$ when we contract the edge contained by $K_2$, we know that each embedding of \VBgce fails to be \iplIIns, by hypotheses $i$ and $ii$.  Thus, no minor of \VBg can be \iplIIns, so \VBg is therefore \mmII.  

\begin{flushright}
$\square$
\end{flushright}


\subsubsection{Subdivisions-Dangle Move} \space
In this sub-section, we will define the Subdivisions-Dangle move, and then show how it preserves the property of being intrinsically minor-minimal intrinsically type II spherical linked.

\medskip

We define a \textit{dangle} in a graph, $G$, to be a $K_2$ with exactly one vertex identified with a subdivision on some edge of $G$ (see Figure 4.6).

We say that $G$ is \textit{dangle-able} if $G$ meets the criterion in Proposition 3.4.

We call $G$ \textit{basic}, with respect to being dangle-able, if $G$ is dangle-able and does not contain any dangles.

\begin{figure}[H]
\centering
\includegraphics[width=0.7\linewidth]{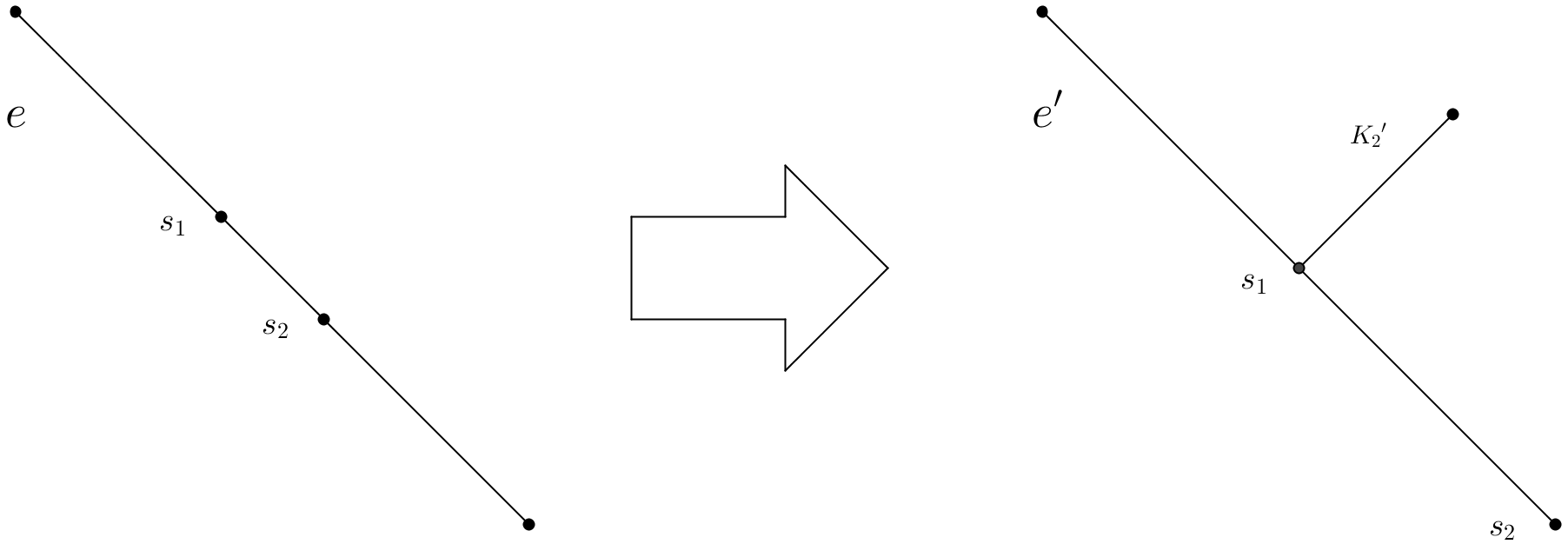}
\caption{An illustration of the Subdivisions-Dangle move.}
\label{fig:SubDangleMove}
\end{figure}

\newpage

\begin{proposition}
Suppose that a graph, $G$, satisfies the following:  \begin{enumerate}[i] 

\item  $G$ is \mmII,

\item  $G$ is a connected planar graph,

\item  $G$ contains an edge, $e$, with two subdivisions $s_1$ and $s_2$ on $e$,

\item  $G \dU \{v\}$ with a contracted edge between $s_2$ and the adjacent endpoint of $e$ is not \iplIIns,  

\item  $G$ with a contracted edge between $s_2$ and the adjacent endpoint of $e$ with a vertex of $K_2$ identified with an endpoint of $e$ is not \iplIIns.

\end{enumerate} 

\vskip -.1in
\noindent Then, the graph $G'$ obtained by contracting the edge between $s_2$ and the adjacent endpoint of $e$ and then identifying a vertex of $K_2$ with $s_1$ is \mmII; we call this particular $K_2$, $K_{2}'$, and we call this new edge including $K_{2}'$, $e'$. 
\end{proposition}

\vskip -.2in
\begin{proof} Suppose that $G$ is a graph that satisfies the properties above.  Consider $G'$.  Notice that, for any spherical embedding, $G'$ and $G$ share the same number of faces and that each face of $G'$ is equivalent to that of $G$.  Further, notice that $e'$ can only be embedded into the equivalent faces of $G'$ that $e$ can be embedded into of $G$.  Thus, $G'$ is \iplIIns, as $G$ is.

Now, we will verify that $G'$ is \mmII.  Consider $G'$ \textit{minus and edge}.  
As $e'$ can only be embedded into the equivalent faces that $e$ can, we see that $G'$ \textit{minus an edge} fails to be \iplII ~whenever we delete an edge of $G'$ that is not contained in $e'$, by hypothesis $i$.  
Furthermore, we notice that deleting an edge contained in $e'$ that is not also contained in $K_{2}'$ is equivalent to deleting an edge between an endpoint of $e$ and an adjacent subdivision of $e$ in $G$, and so $G'$ \textit{minus an edge} fails to be \iplII ~whenever we delete one of these two edges, also by hypothesis $i$.  Finally, $G'$ \textit{minus the edge contained in} $K_{2}'$ is the same graph as described in hypothesis $iii$, so in general, $G'$ \textit{minus an edge} fails to be \iplIIns.  
Thus, deleting any vertex of $G'$ that is connected to an edge fails to be \iplIIns.  Also, since $e'$ can only be embedded into the equivalent faces that $e$ can, $G'$ \textit{minus a disjoint vertex (if $G'$ has one)} fails to be \iplII ~in the same way that $G$ would fail, by hypothesis $i$.  So, in general, $G'$ \textit{minus a vertex} fails to be \iplIIns.  Lastly, consider $G'$ \textit{with a contracted edge}.  
Similar to above, as $e'$ can only be embedded into the equivalent faces that $e$ can, we see that $G'$ \textit{with a contracted edge} fails to be \iplII ~whenever we contract an edge of $G'$ that is not contained in $e'$, by hypothesis $i$.  Furthermore, we notice that contracting an edge contained in $e'$ that is not also contained in $K_{2}'$ is the same graph as described in hypothesis $iv$, so $G'$ \textit{with a contracted edge} fails to be \iplII ~whenever we contract one of these two edges.  Lastly, $G'$ \textit{with the edge contained in} $K_{2}'$ \textit{contracted} is a minor of $G$, and so in general, $G'$ \textit{with a contracted edge} fails to be \iplIIns, by hypothesis $i$.  Thus, no minor of $G'$ can be \iplIIns, so $G'$ is therefore \mmII. 
\end{proof}


\begin{figure}[H]
\centering
\includegraphics[width=0.9\linewidth]{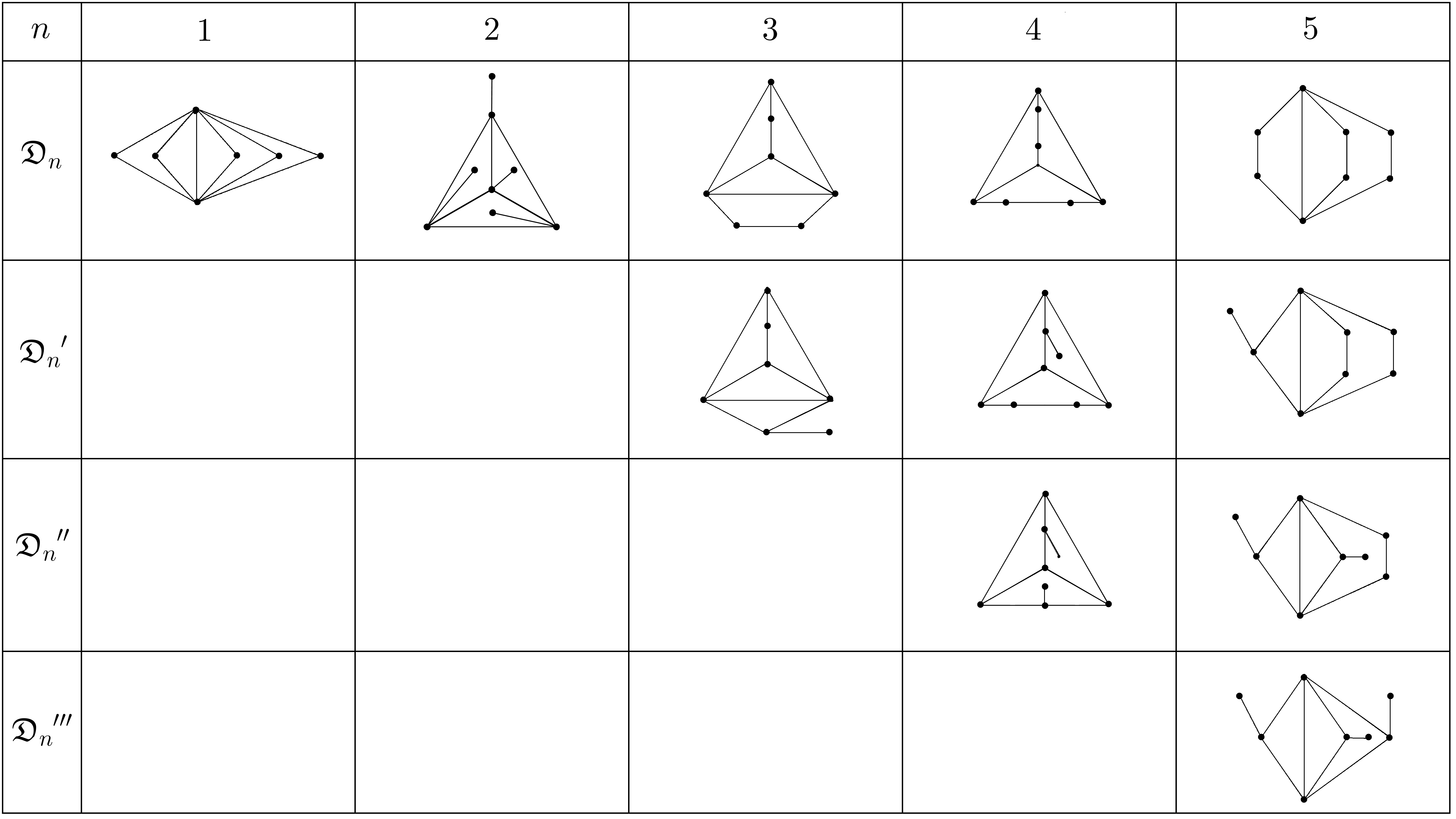}
\caption{[Top Row]  Embeddings of our basic minor-minimal intrinsically type II spherical 3-linked graphs (with the exceptions of \Dons and \Dtw).  [Lower Rows]  Embeddings of our basic graphs with Sub-Dangle moves applied to them.}
\label{fig:Handout1}
\end{figure}

\vskip -.35in

\begin{figure}[H]
\centering
\includegraphics[width=0.9\linewidth]{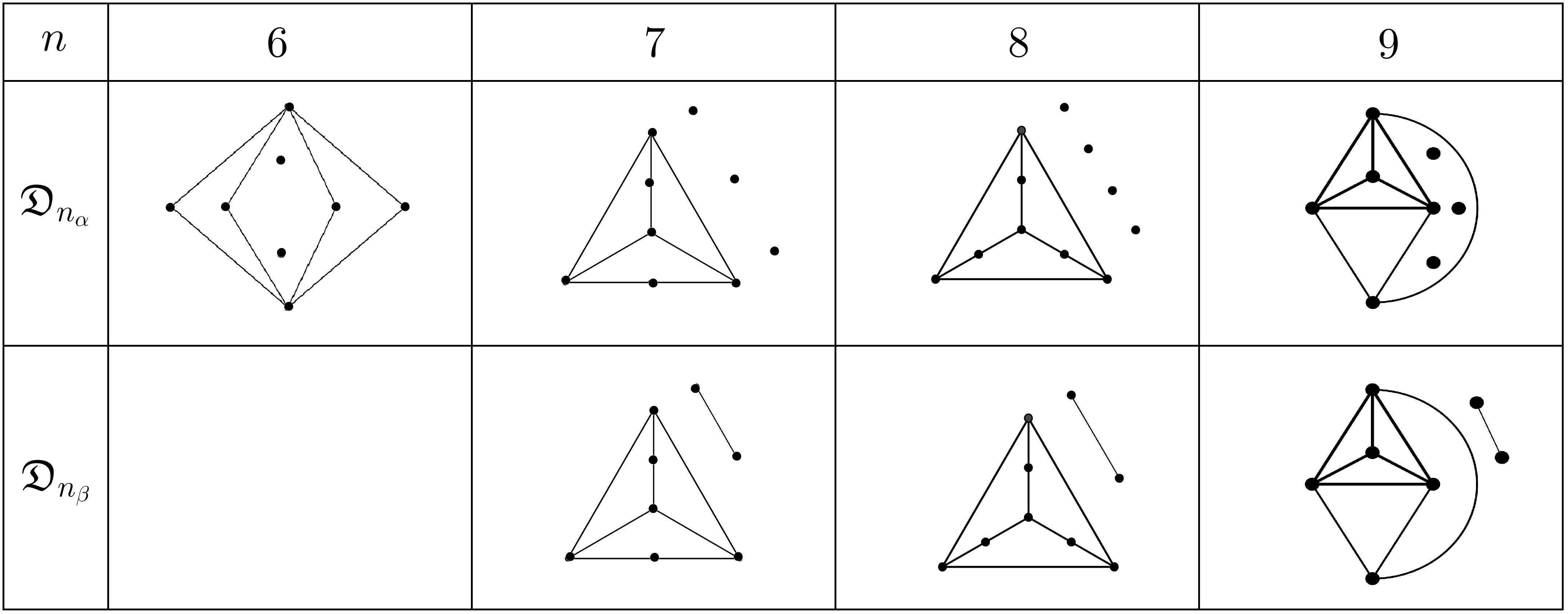}
\label{fig:Handout2}
\end{figure}

\begin{figure}[H]
\centering
\includegraphics[width=0.9\linewidth]{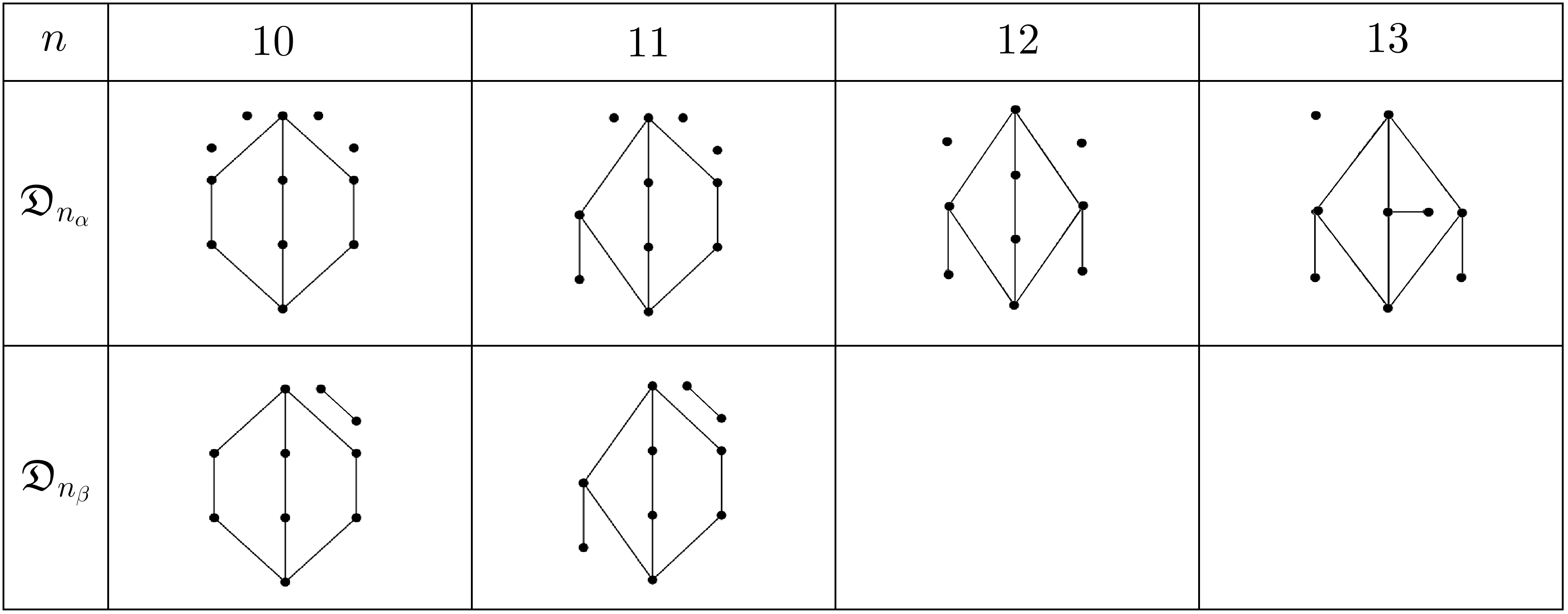}
\caption{[Top Rows]  Embeddings of our alpha versions of our minor-minimal intrinsically type II spherical 3-linked graphs (with the exceptions of \Dsi, \Dtwe, and \Dthi.)  [Bottom Rows]  Embeddings of our beta versions of our minor-minimal intrinsically type II spherical 3-linked graphs.}
\label{fig:Handout3}
\end{figure}

\begin{proposition}
The graphs \Don, \Dtw, \Dth, \Dfo, \Dfi, \Dsi, \Dsea, \Deia, \Dnia, \Dtena, \Deleb, \Dtwe, and \Dthis are \mmII.
\end{proposition}

 \noindent
 \textit{Proof.} \space  First, notice that the graphs $\mathfrak{D}_i$ for $i \in \mathbb{N}$ with $ 1 \leq i \leq 5$ are connected, so it suffices to show that deleting or contracting any edge from these graphs will result in a loss of our graph being intrinsically type II spherical 3-linked, in order to show minor-minimality (since we will not be able to delete a vertex from these graphs without also deleting an edge).  
Similarly, notice that we can characterize the vertices in the graphs $\mathfrak{D}_j$ for $j \in \mathbb{N}$ with $6 \leq j \leq 13$ as being incident to an edge or not.  
With this characterization, it suffices to show that deleting or contracting any edge from these graphs, or deleting vertex component from these graphs, will result in a loss of our graph being intrinsically type II spherical 3-linked, in order to show minor-minimality.

First, we embed $K_4$ into \Sph.  This embedding is unique, up to equivalence, as $K_4$ is 3-connected.  
Additionally, this embedding determines four faces, each of which are equivalent, and we also note that each edge and vertex in $K_4$ are pairwise equivalent.  
In regards to Figure \ref{fig:D2}, we see that by identifying just one vertex of a copy of $K_2$ to each vertex of $K_4$, we are forced to embed the other vertex of each copy of $K_2$ into a face of $K_4$, although at most three remaining vertices can be embedded into one face of $K_4$.  
Since the faces of $K_4$ are all equivalent, we embed one vertex of a copy of $K_2$ into any face of $K_4$; call this face $F$.  As each face is adjacent to this face, we are left with only four possible ways to embed the last three vertices:
\begin{enumerate}
\item Embed exactly two of the remaining vertices into $F$, 

\item Embed exactly one of the remaining vertices into $F$, and embed the last two remaining vertices into one distinct face, 

\item Embed exactly one of the remaining vertices into $F$, and embed the last two remaining vertices into two distinct faces.

or

\item Embed exactly one of the remaining vertices into each of the remaining faces.
\end{enumerate} 

\noindent Thus, the graph \Dtw ~has only four distinct embeddings into \Sph, up to equivalence. 
\noindent 
We note that in a similar way, $K_4$ \textit{with one subdivision on two nonadjacent edges} and $K_4$ \textit{with one subdivision on three pairwise adjacent edges} have a unique embedding into \Sph, which determines four equivalent faces.  

In regards to the graph \Dsea, as in Figure \ref{fig:D7a}, we see that after embedding one vertex into a face $F_1$ of $K_4$ \textit{with one subdivision on two nonadjacent edges}, we are only left with the options (1), (3) (minus a vertex), and (4).  

\noindent In regards to the graph \Deia, as in Figure \ref{fig:D8a}, we see that after embedding one vertex into a face $F_2$ of $K_4$ \textit{with one subdivision on three pairwise adjacent edges}, we have all of the four options above to place our remaining vertices, as well as a fifth option to place all three of the remaining vertices into $F_2$.

Now, regarding the graph \Dth, as in Figure \ref{fig:D3}, we see that subdividing an edge of $K_4$ yields two equivalence classes of faces; those that are bounded by an edge with a subdivision, and those that are not.  So, when we identify the two endpoints of a copy of $K_2$ \textit{with two subdivisions} to the two vertices of $K_4$ \textit{with one subdivision} that are not incident to an edge with a subdivision, we see that we are forced to embed this copy of $K_2$ \textit{with two subdivisions} into one of  two equivalent faces of $K_4$ \textit{with a subdivision}; a face that is not bounded by an edge with a subdivision.  Thus, \Dths has a unique embedding into \Sph, up to equivalence.  

Hence, as can be seen in Figures \ref{fig:D2}, \ref{fig:D3}, \ref{fig:D7a}, and \ref{fig:D8a}, \Dtw, \Dth, \Dsea, and \Deias are \iplII.

Now, consider the graph \Dsi. First, we embed $K_{4,2}$ into \Sph.  Because of the symmetry of $K_{4,2}$, this embedding is unique, up to equivalence.  Also, this embedding determines four faces, each of which are equivalent.  In regards to Figure \ref{fig:D6}, we see that there are only three ways to embed two disjoint vertices into faces of $K_{4,2}$ in \Sph; embedding both vertices into one face, embedding one vertex into two adjacent faces, or embedding one vertex into two nonadjacent face.  Thus, there are three distinct ways to embed \Dsis into \Sph.  Hence, as seen in Figure \ref{fig:D6}, \Dsis is \iplII.  

Consider the graph \Dnia.  We embed $K_5 - e$ into \Sph.  This embedding is unique, up to equivalence, as $K_5 - e$ is 3-connected, and determines six faces which are equivalent.  
This leaves us with three options; embed all three vertex components into one face, embed exactly two vertex components into one face, or embed exactly one vertex component into any face.  
Moreover, when we embed exactly two vertex components into one face of $K_5 - e$, we have sub-options to embed the last vertex component into into one of two adjacent faces (these two faces are distinct since two vertices have been embedded into one particular face), or into a nonadjacent face.  Furthermore, when we embed exactly one vertex into any face, we have sub-options to have one vertex in three pairwise adjacent faces, one vertex in three non-pairwise adjacent faces, or one vertex in two adjacent faces and one vertex into a pairwise nonadjacent face.  Additionally, when we embed one vertex into three non-pairwise adjacent faces, we have two sub-sub-options when we choose how to place the three vertices.  Thus, \Dnias has a total of eight distinct embeddings into \Sph.  
Hence, as seen in Figure \ref{fig:D9a}, \Dnias is \iplII.

Next, consider the graph \Dtena.  We embed $K_{3,2}$ \textit{with one subdivision on three pairwise adjacent edges} into \Sph.  Because of the symmetry of $K_{3,2}$ \textit{with one subdivision on three pairwise adjacent edges}, this embedding is unique, up to equivalence, and determines three equivalent faces.  In regards to Figure \ref{fig:D10a}, we see that there are only four ways to embed our remaining four disjoint vertices; embed exactly four vertices into one face, embed exactly three vertices into one face, embed exactly two vertices into any face, or embed exactly two vertices into exactly one face.  Thus, \Dtenas has four distinct embeddings into \Sph.  Hence, by Figure \ref{fig:D10a}, \Dtenas is \iplII.

Now, consider the graph \Delea.  In regards to Figure \ref{fig:D11a}, we see that \Delea \textit{minus three disjoint vertices} has a unique embedding into \Sph, as the only vertex left that is not contained by $K_{3,2}$ can only be embedded into one of two equivalent faces of \Deleas \textit{minus three disjoint vertices}.  So, we embed this particular vertex into any one of these two possible faces; call this face $H$.  Now, we notice that whenever we embed another vertex into $H$, the resultant embedding is type II spherical 3-linked.  Similarly, whenever we embed any two vertices into one face, we have that the resultant embedding is type II spherical 3-linked.  Hence, as there are only three faces in \Deleas \textit{minus three disjoint vertices}, we have that \Deleas is \iplII.

Continuing, consider the graph \Dtwe.  In regards to Figure \ref{fig:D12}, we see that \Dtwes \textit{minus two disjoint vertices} has two distinct embeddings into \Sph, as the only two vertices left that are not contained by $K_{3,2}$ can be embedded into a common face, or not.  Notice, when \Dtwes \textit{minus two disjoint vertices} is embedded with these two vertices in a common face, the resultant embedding is type II spherical 3-linked regardless of where we embed the two disjoint vertices.  Furthermore, when we embed \Dtwes \textit{minus two disjoint vertices} with these two vertices in adjacent faces, we have that whenever we embed one vertex into one of these two faces, or whenever we embed both disjoint vertices into a common face, that the resultant embedding will be type II spherical 3-linked.  Thus, we have that \Dtwes is \iplII.

Lastly, consider the graph \Dthi.  In regards to Figure \ref{fig:D13}, we see that \Dthis \textit{minus one disjoint vertex} has only two distinct embeddings into \Sph, up to equivalence, one of which is type II spherical 3-linked no matter where another vertex is added.  Moreover, observing Figure \ref{fig:D13}, we see that no matter where we place a vertex in any embedding of \Dthis \textit{minus one disjoint vertex} results in an embedding that is type II spherical 3-linked.  Thus, \Dthis is \iplII.

Finally, for each $1 \leq n \leq 6$ and each $7 \leq m \leq 13$, we note that the graphs $D_n$ and $D_{m_\alpha}$ are minor-minimal with respect to being intrinsically type II spherical 3-linked, as can be seen in Figures $n$B and $m$B, respectively.

Our proof of Proposition 3.5 is now complete.  
\begin{flushright}
$\square$
\end{flushright}

\begin{corollary}
The graphs \Dth$'$, \Dfo$'$, \Dfo$''$, \Dfi$'$, \Dfi$''$, \Dfi$'''$, \Dseb, \Deib, \Dnib, \Dtenb, and \Delebs are \mmII.
\end{corollary}

 \noindent
 \textit{Proof.} \space Recall, the Sub-Dangle and Vert-Bar moves preserve minor-minimality of intrinsic type II spherical 3-linkings when certain criteria are met. 

By Proposition 3.5, \Dth, \Dfo, and \Dfi fulfill requirement $i$ of the Sub-Dangle move.  
By examining Figures \ref{fig:D3}, \ref{fig:D4}, and \ref{fig:D5}, we see that each of these graphs also fulfills requirement $ii$.  
Lastly, by examining Figures \ref{fig:D3subdangle}, \ref{fig:D4subdangle}, and \ref{fig:D5subdangle},  we see that each of these graphs fulfill requirements $iii$ and $iv$, as well.  Thus, \Dth$'$, \Dfo$'$, and \Dfi$'$ are \mmII.

Now, we see that \Dfo$'$ and \Dfi$'$ fulfill requirements $i$ and $ii$ of the Sub-Dangle move.  By examining Figures \ref{fig:D4'subdangle} and \ref{fig:D5'subdangle}, we see that these two graphs fulfill requirements $iii$ and $iv$, as well.  Thus, \Dfo$''$ and \Dfi$''$ are \mmII. Similarly, by examining Figure \ref{fig:D5''subdangle}, wee see that \Dfi$'''$ is \mmII.

Continuing, by Proposition 3.5 we have that \Dsea, \Deia, \Dnia, \Dtena, and \Deleas fulfill requirement $i$ of the Vert-Bar move.
By examining Figures \ref{fig:D7a}, \ref{fig:D8a}, \ref{fig:D9a}, \ref{fig:D10a}, and \ref{fig:D11a}, we see that these graphs also fulfill requirements $ii$ and $iii$.  
Finally, by examining Figures \ref{fig:D7aminors}, \ref{fig:D8aminors}, \ref{fig:D9aminors}, \ref{fig:D10aminors}, and \ref{fig:D11aminors}, we see that these graphs fulfill requirements $iv$ and $v$, as well.  Thus, \Dseb, \Deib, \Dnib, \Dtenb, and \Deleb are \mmII.

This completes our proof.  
\begin{flushright}
$\square$
\end{flushright}


%
%
%
%
%
%


\vskip .2in
 
\begin{figure}[H]
\centering
\includegraphics[width=0.35\linewidth, height=0.1\textheight]{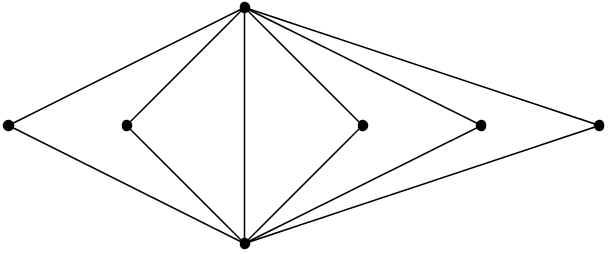}
\caption{\Dons has a unique embedding into \Sph, up to equivalence.}
\label{fig:D1}
\end{figure}

\begin{figure}[H]
\centering
\includegraphics[width=0.7\linewidth, height=0.2\textheight]{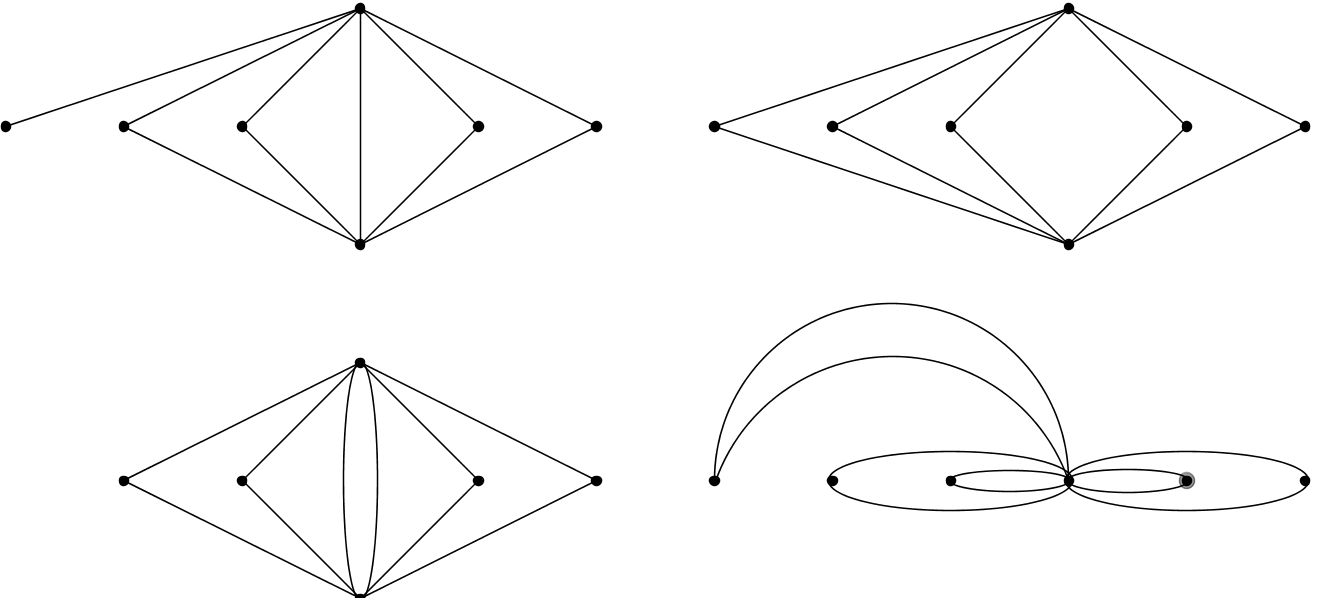}
\caption{\Dons has only two equivalence classes of edges; those that are incident to a vertex of degree 2, and those that are not.}
\label{fig:D1minors}
\end{figure}

\begin{figure}[H]
\centering
\includegraphics[width=0.7\linewidth, height=0.2\textheight]{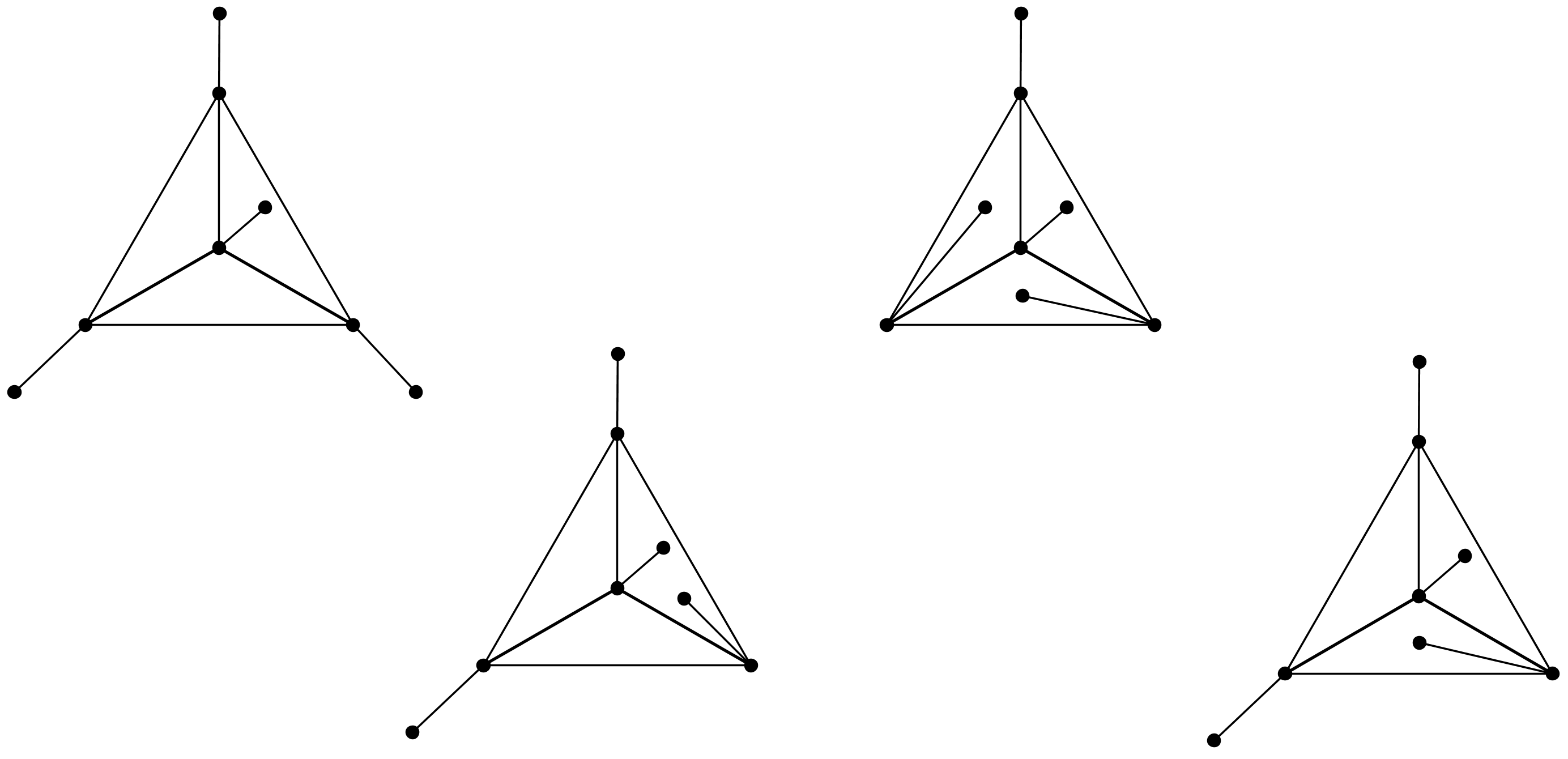}
\caption{\Dtws has four distinct embeddings into \Sph, up to equivalence.}
\label{fig:D2}
\end{figure}

\begin{figure}[H]
\centering
\includegraphics[width=0.4\linewidth, height=0.15\textheight]{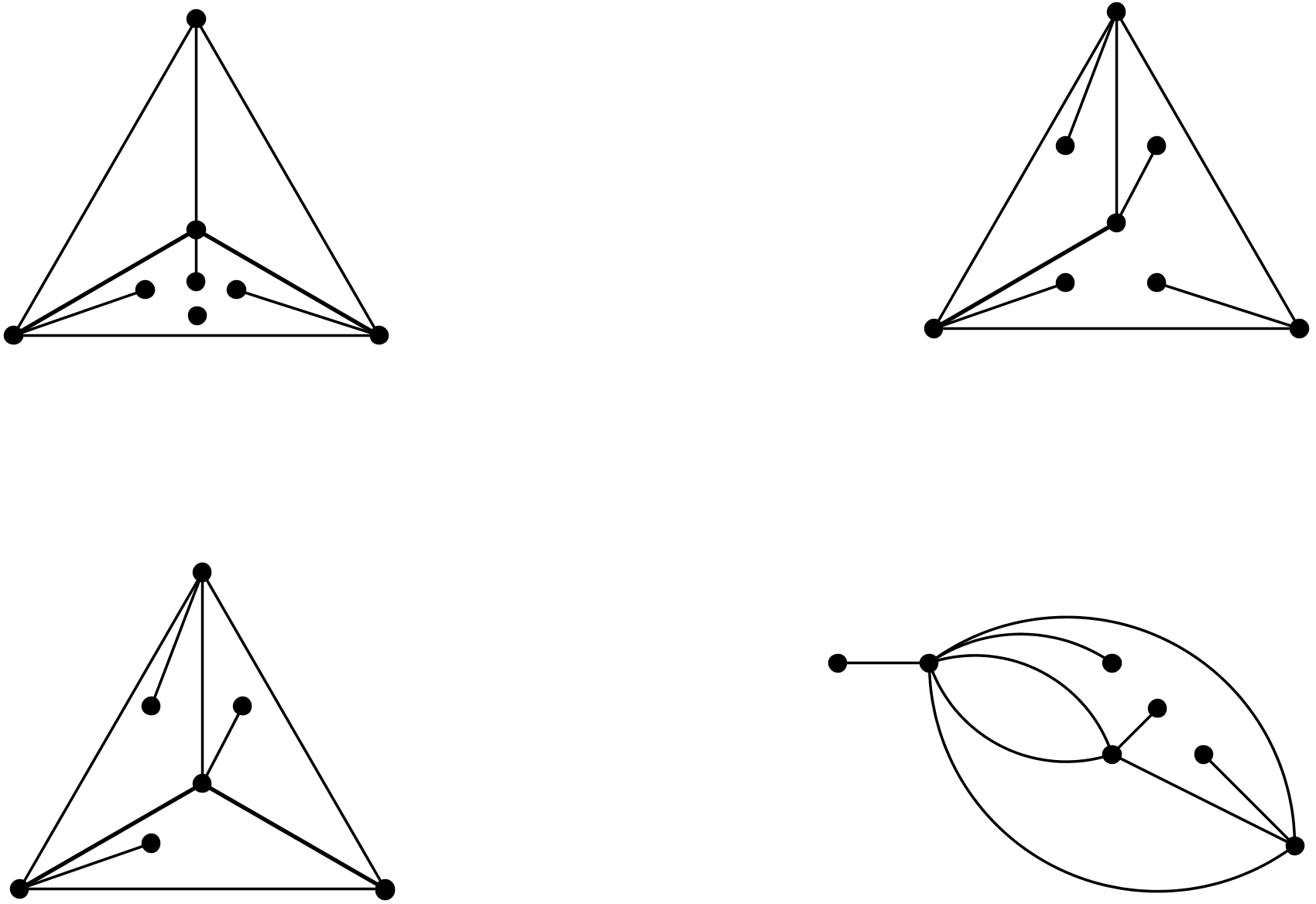}
\caption{\Dtws has only two equivalence classes of edges; those that are contained by $K_4$, and those that are not.}
\label{fig:D2minors}
\end{figure}

\begin{figure}[H]
\centering
\includegraphics[width=0.5\linewidth, height=0.1\textheight]{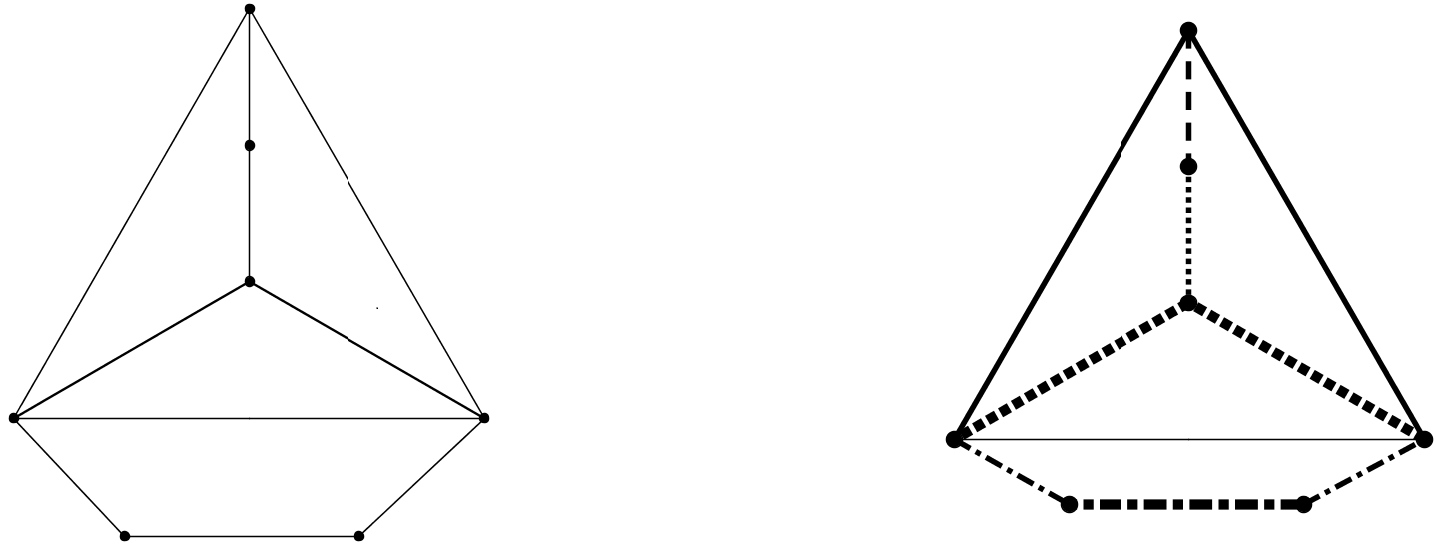}
\caption{\Dths has a unique embedding into \Sph, up to equivalence.}
\label{fig:D3}
\end{figure}

\begin{figure}[H]
\centering
\includegraphics[width=0.8\linewidth, height=0.25\textheight]{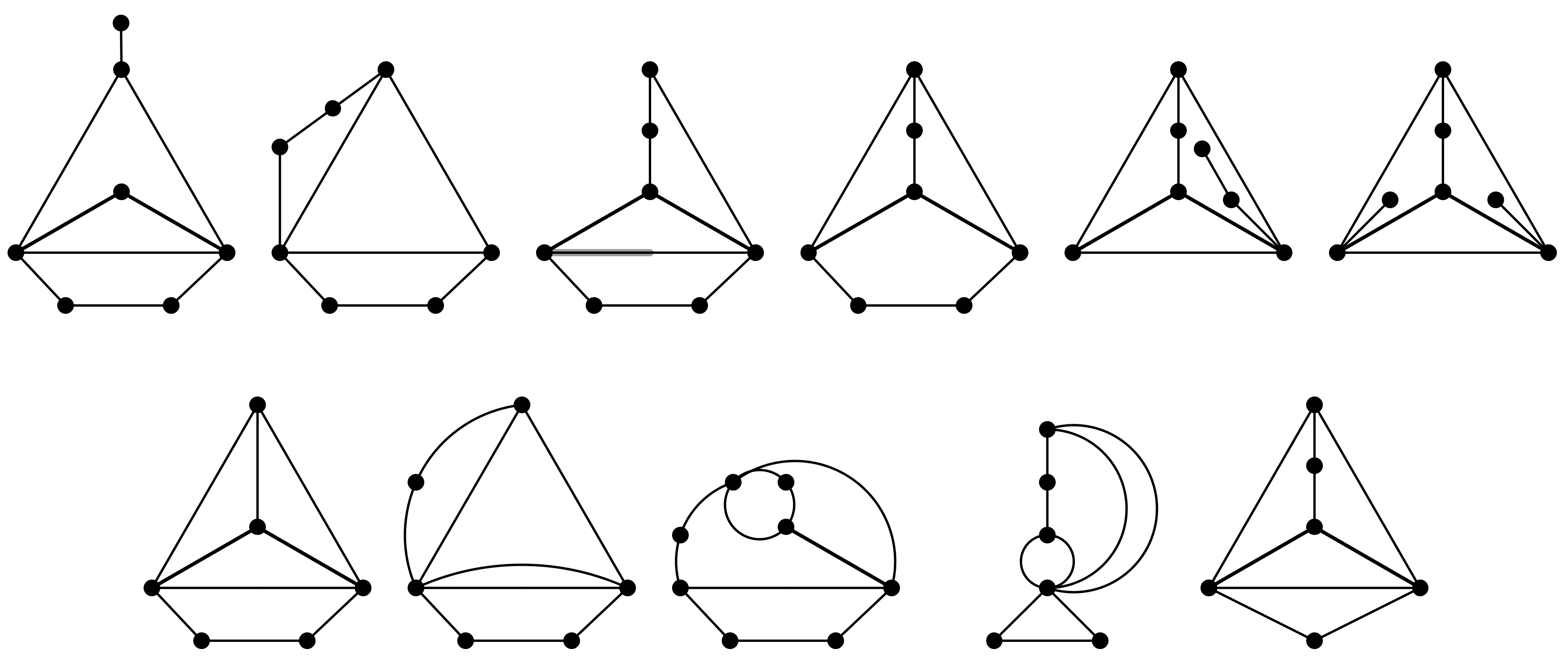}
\vskip .1in
\caption{\Dths has seven equivalence classes of edges, as highlighted in Figure \ref{fig:D3}.}
\label{fig:D3minors}
\end{figure}

\begin{figure}[H]
\centering
\vskip -.25in
\includegraphics[width=0.2\linewidth, height=0.1\textheight]{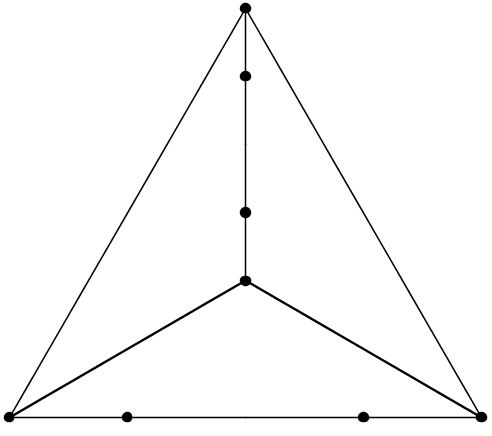}
\vskip .1in
\caption{\Dfos has a unique embedding into \Sph, up to equivalence, since \Dfos is 3-connected (disregarding subdivisions).}
\label{fig:D4}
\end{figure}

\begin{figure}[H]
\centering
\vskip -.2in
\includegraphics[width=0.7\linewidth, height=0.2\textheight]{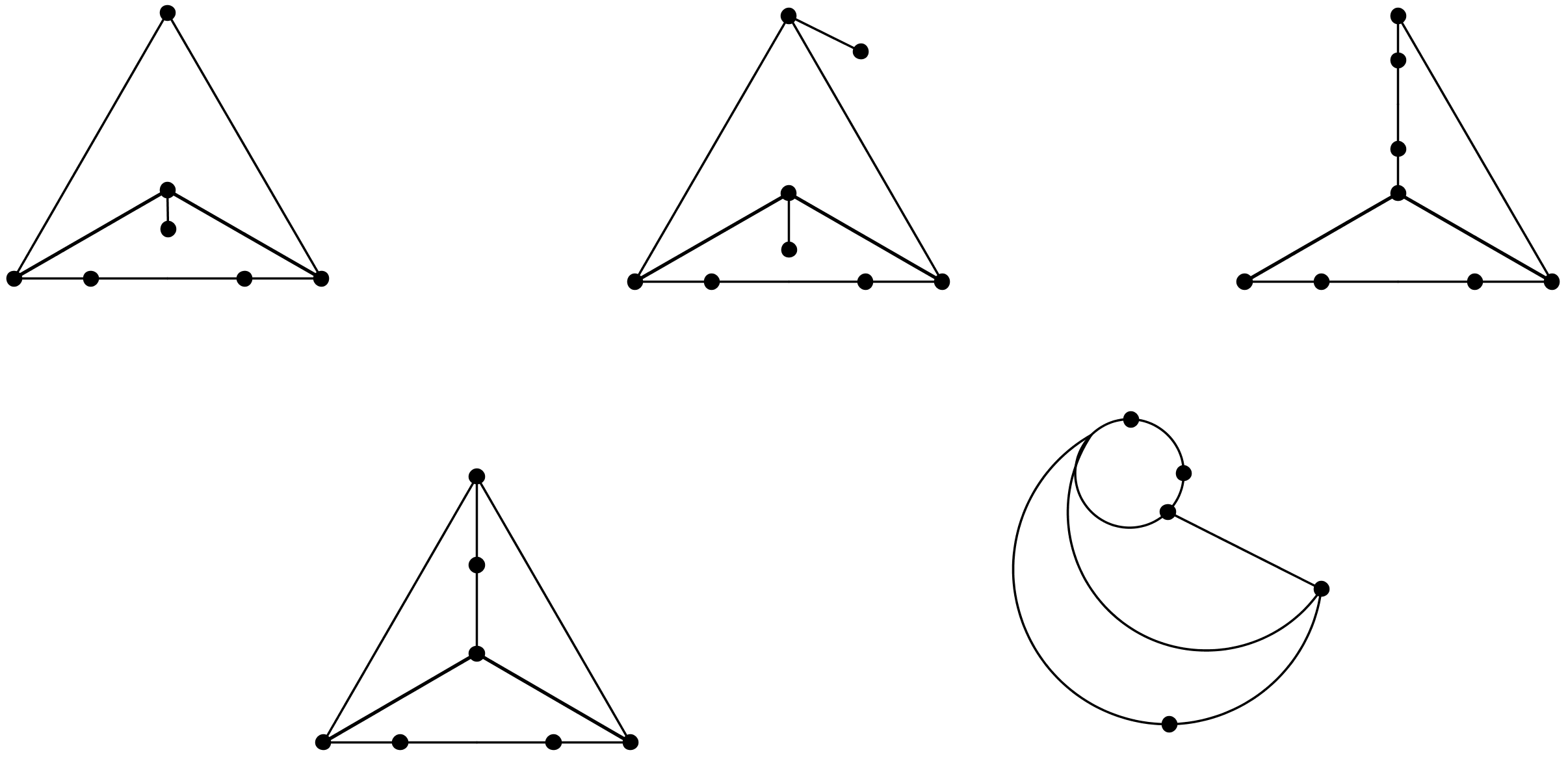}
\vskip .1in
\caption{\Dfos has three equivalence classes of edges; those that are incident to two vertices of degree 2, those that are incident to two vertices of degree 3, and those that are neither.}
\label{fig:D4minors}
\end{figure}

\begin{figure}[H]
\centering
\includegraphics[width=0.2\linewidth, height=0.1\textheight]{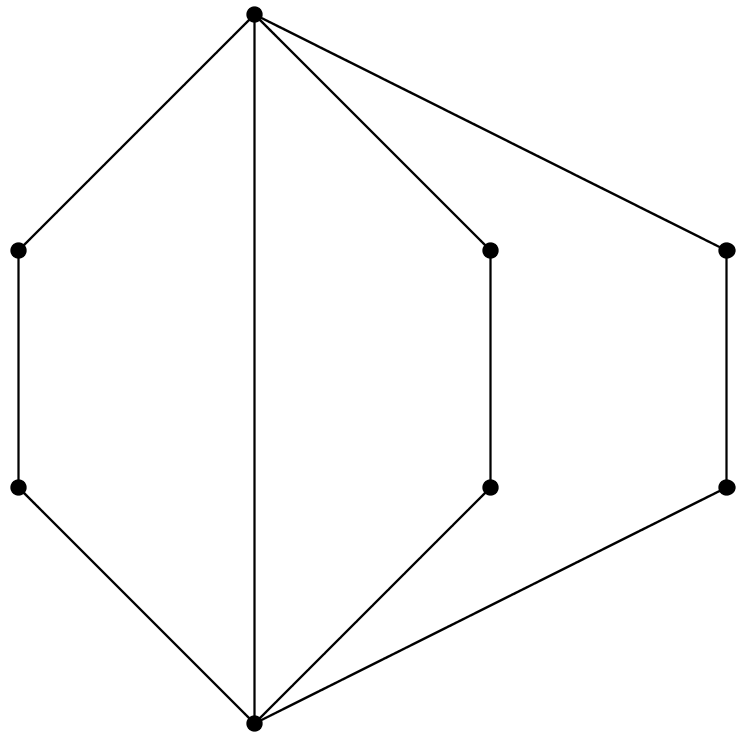}
\caption{\Dfis has a unique embedding into \Sph, up to equivalence, by the symmetery of $K_{3,1,1}$ \textit{with one subdivision on three pairwise adjacent edges}.}
\label{fig:D5}
\end{figure}

\begin{figure}[H]
\centering
\includegraphics[width=0.7\linewidth, height=0.2\textheight]{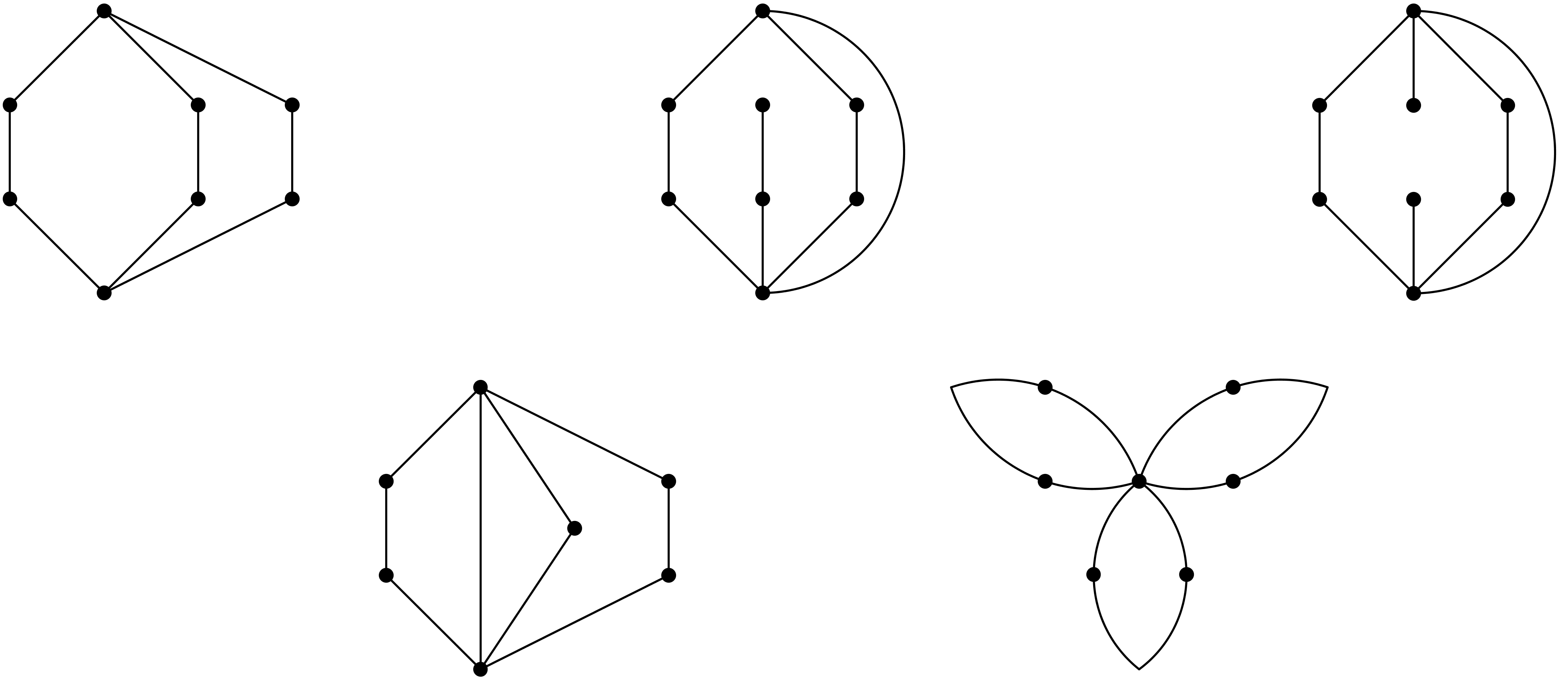}
\caption{\Dfis has three equivalence classes of edges; that that are incident to two vertices of degree 2, those that are incident to two vertices of degree 4, and those that are neither.}
\label{fig:D5minors}
\end{figure}

\begin{figure}[H]
\centering
\includegraphics[width=0.68\linewidth, height=0.1\textheight]{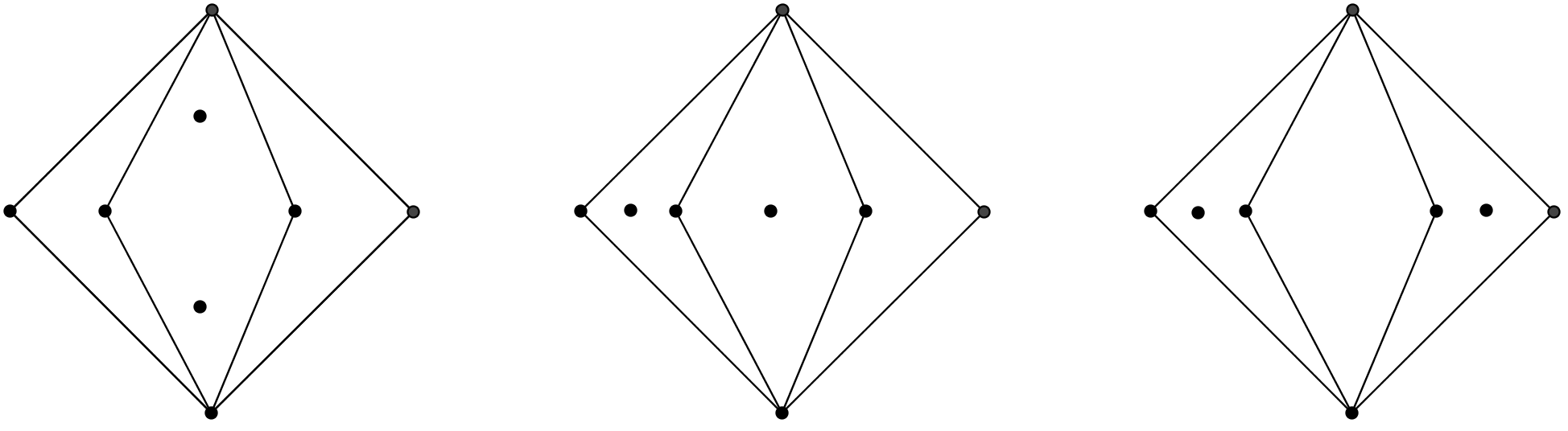}
\caption{\Dsis has three distinct embeddings in \Sph, up to equivalence.}
\label{fig:D6}
\end{figure}

\begin{figure}[H]
\centering
\includegraphics[width=0.68\linewidth, height=0.1\textheight]{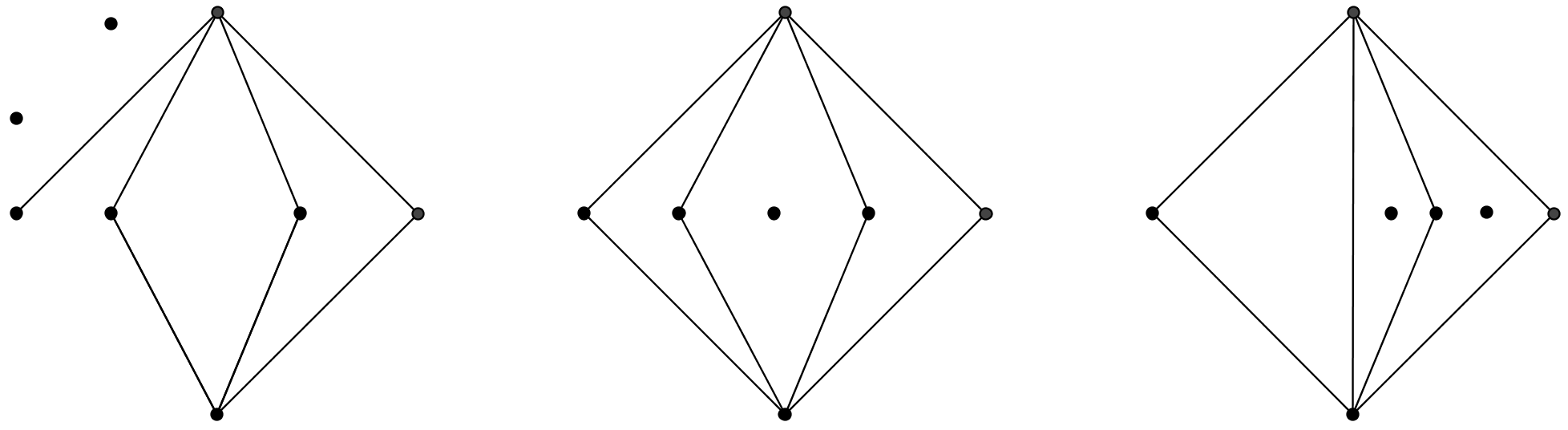}
\caption{Every edge in \Dsis is equivalent, based on the symmetry of $K_{4,2}$.}
\label{fig:D6minors}
\end{figure}

\begin{figure}[H]
\centering
\includegraphics[width=0.7\linewidth, height=0.1\textheight]{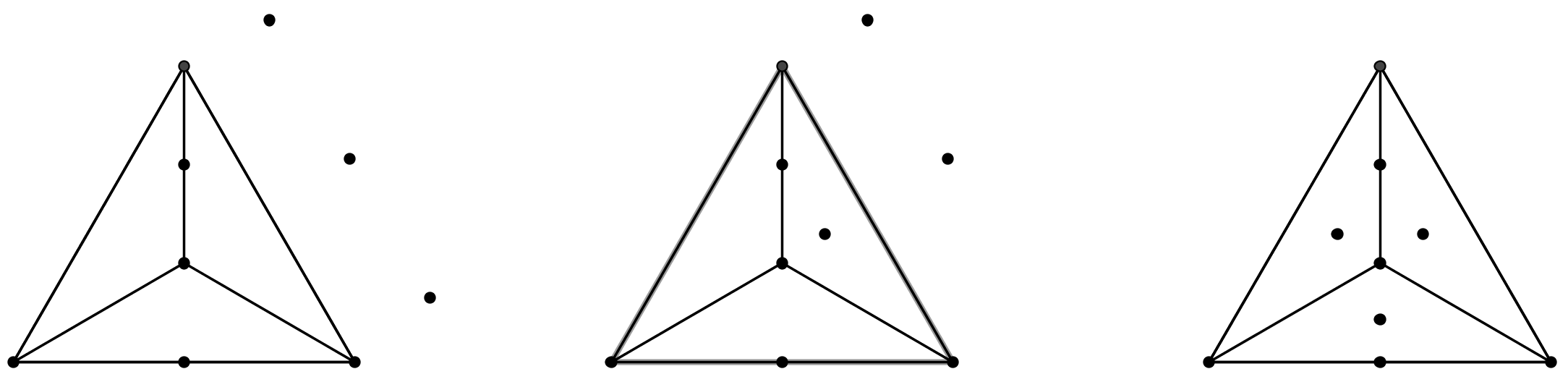}
\caption{\Dseas has three \demb.}
\label{fig:D7a}
\end{figure}

\vskip -.25in

\begin{figure}[H]
\centering
\includegraphics[width=0.7\linewidth, height=0.2\textheight]{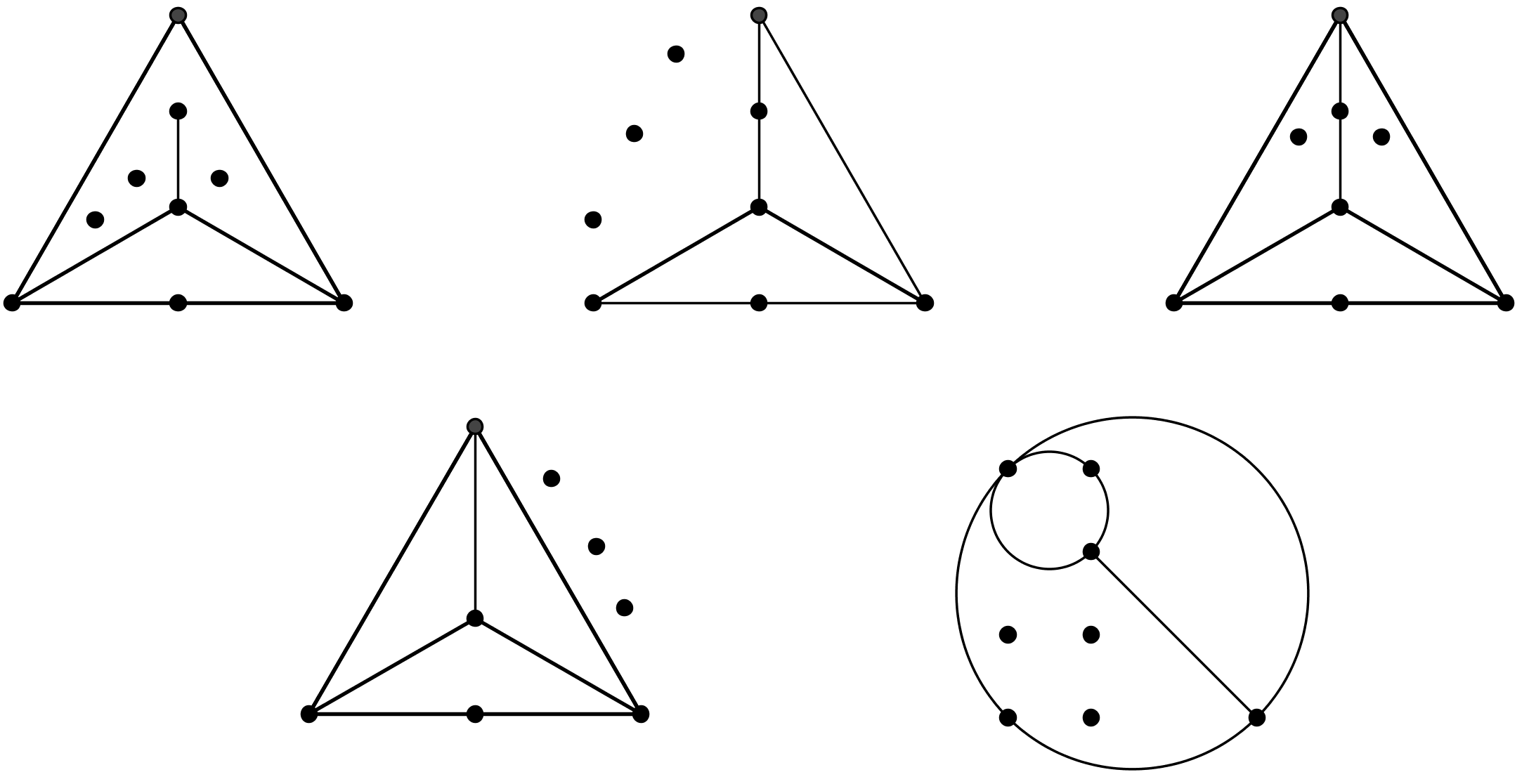}
\caption{\Dseas has only two \equi; those that are incident to two vertices of degree 3, and those that are not.}
\label{fig:D7aminors}
\end{figure}

\vskip -.3in
\begin{figure}[H]
\centering
\includegraphics[width=0.7\linewidth, height=0.2\textheight]{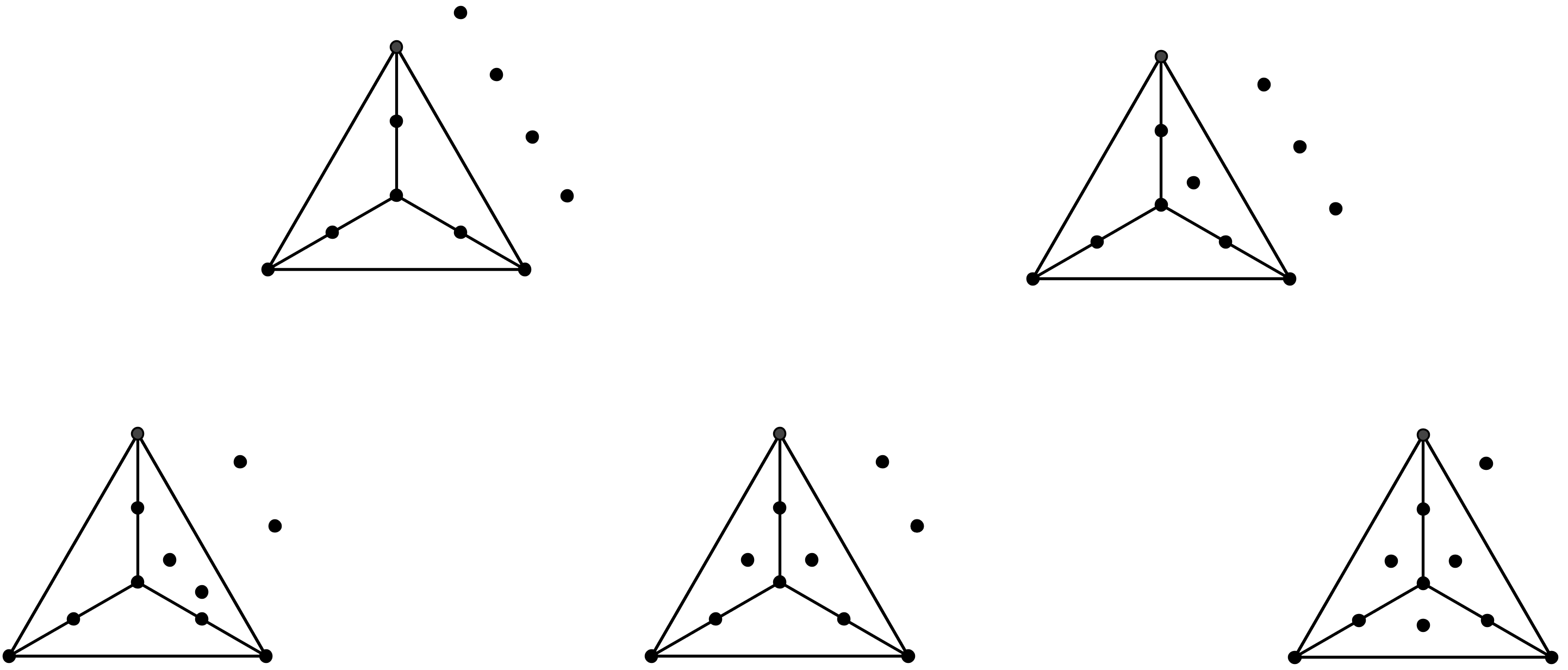}
\caption{\Deias has five \demb.}
\label{fig:D8a}
\end{figure}

\begin{figure}[H]
\centering
\includegraphics[width=0.7\linewidth, height=0.2\textheight]{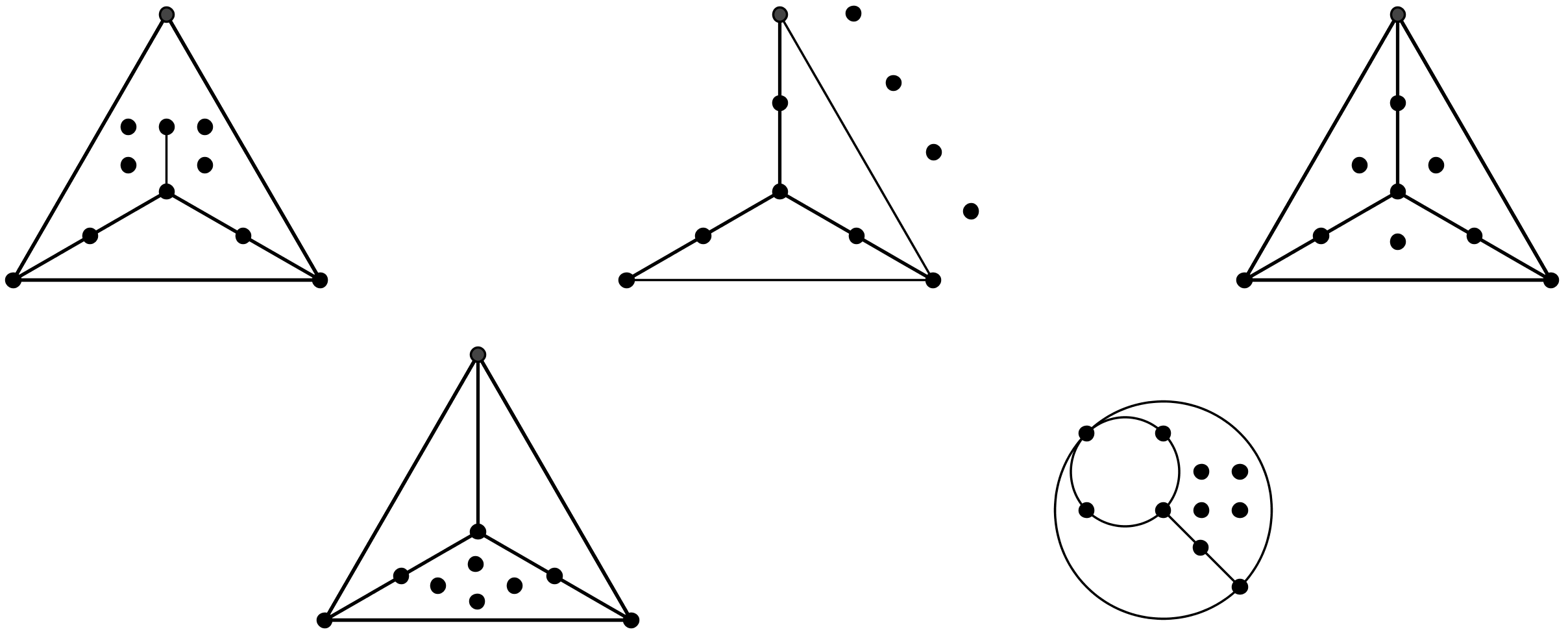}
\caption{\Deias has only two \equi; those that are incident to two vertices of degree 3, and those that are not.}
\label{fig:D8aminors}
\end{figure}

\begin{figure}[H]
\centering
\includegraphics[width=0.9\linewidth, height=0.25\textheight]{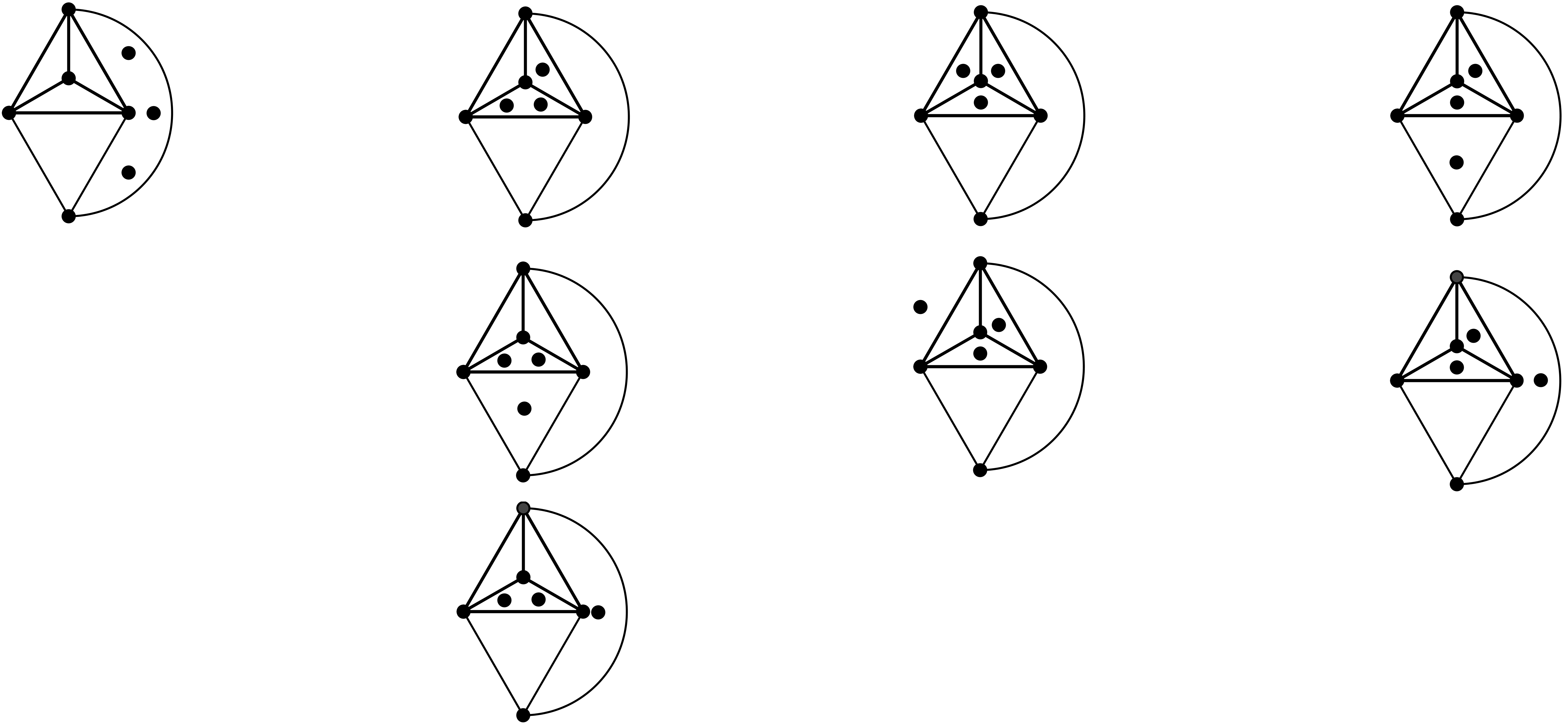}
\caption{\Dnias has eight \demb.}
\label{fig:D9a}
\end{figure}

\vskip -.225in

\begin{figure}[H]
\centering
\includegraphics[width=0.7\linewidth, height=0.2\textheight]{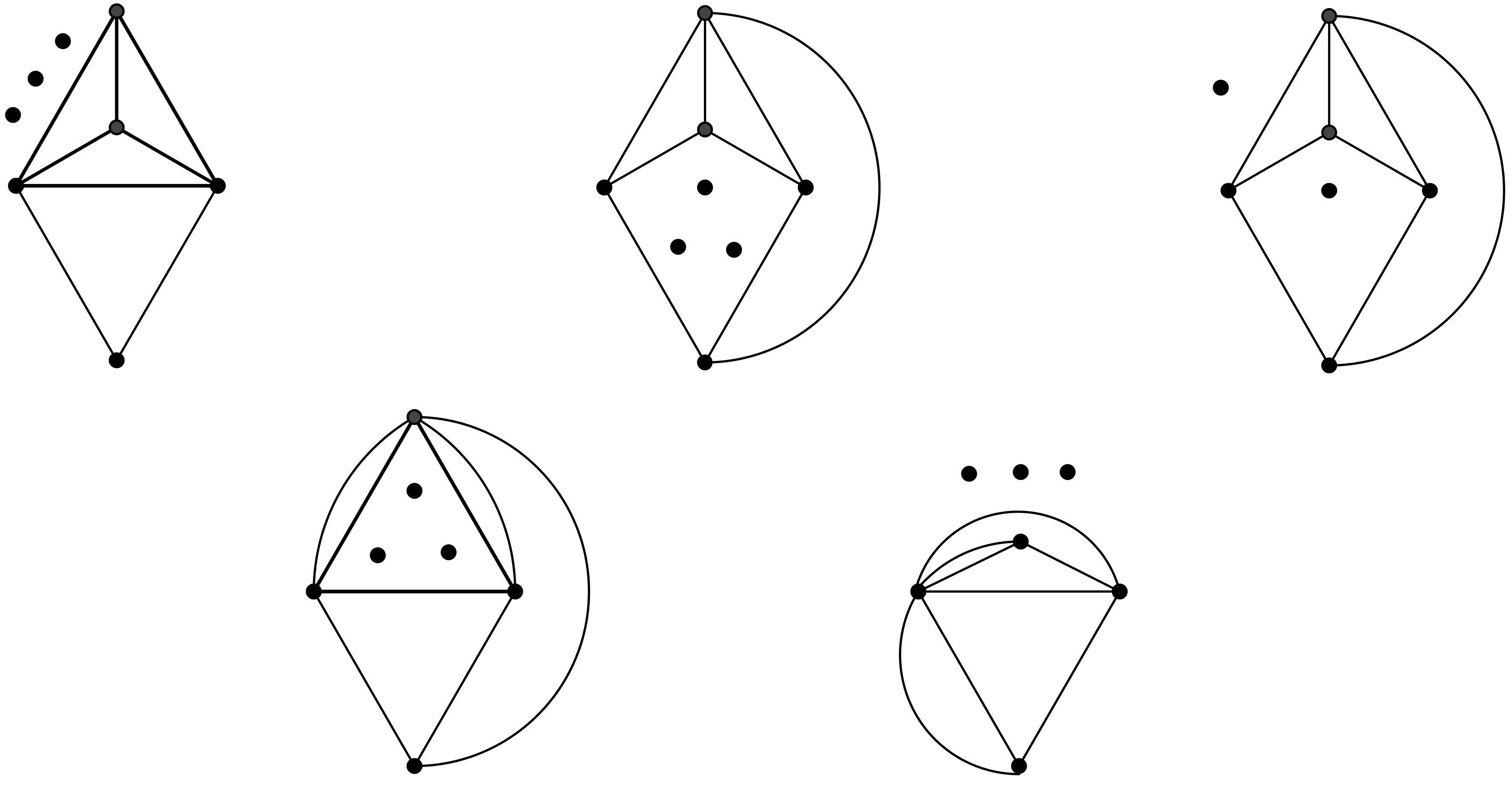}
\caption{\Dnias has only two \equi; those that are incident to two vertices of degree 4, and those that are not.}
\label{fig:D9aminors}
\end{figure}

\vskip -.225in
\begin{figure}[H]
\centering
\includegraphics[width=0.7\linewidth, height=0.08\textheight]{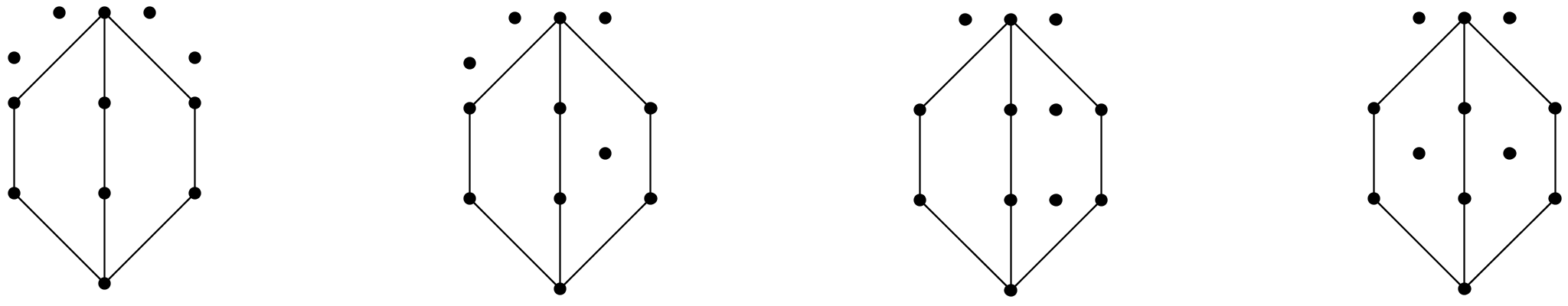}
\caption{\Dtenas has four \demb.}
\label{fig:D10a}
\end{figure}

\begin{figure}[H]
\centering
\includegraphics[width=0.4\linewidth, height=0.2\textheight]{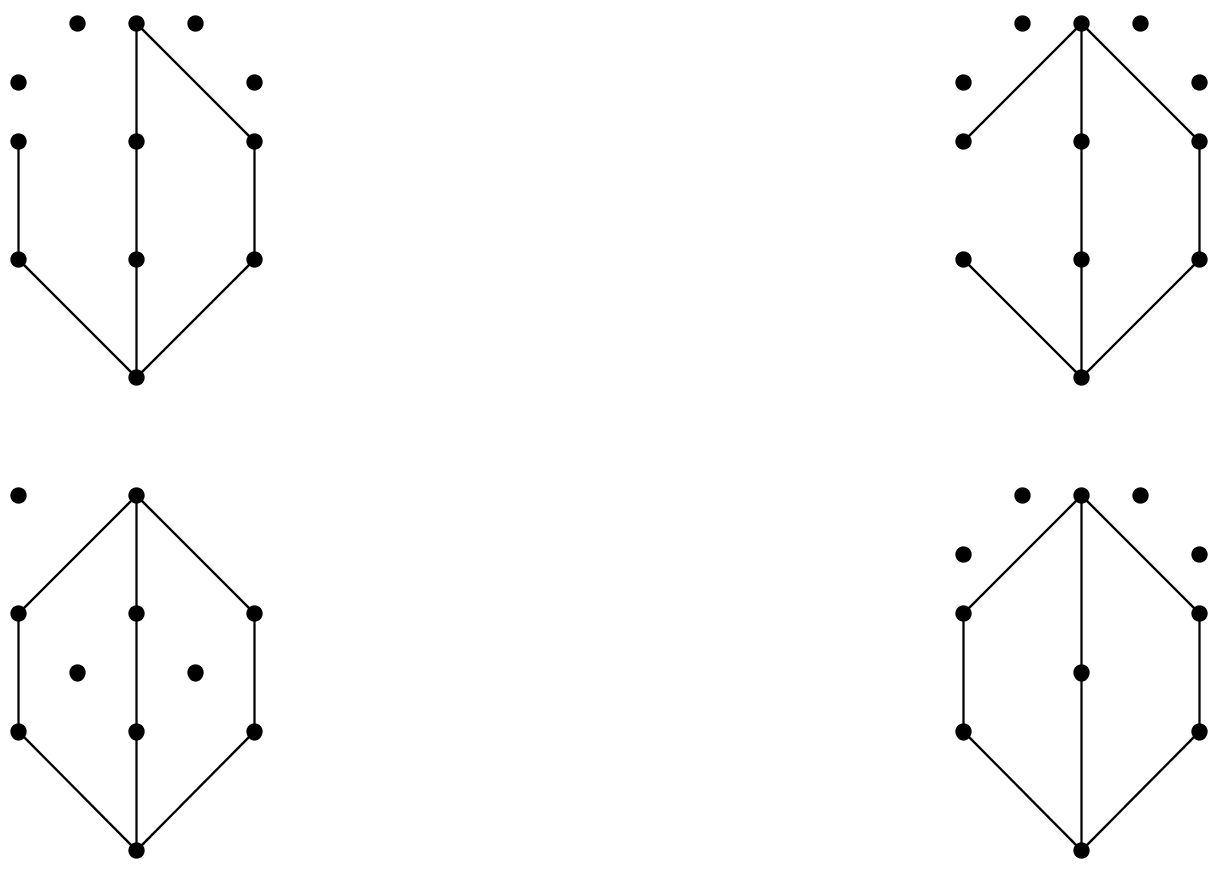}
\caption{\Dtenas has only two \equi; those that are incident to two vertices of degree 2, and those that are not.}
\label{fig:D10aminors}
\end{figure}

\begin{figure}[H]
\centering
\includegraphics[width=.65\linewidth, height=.2\textheight]{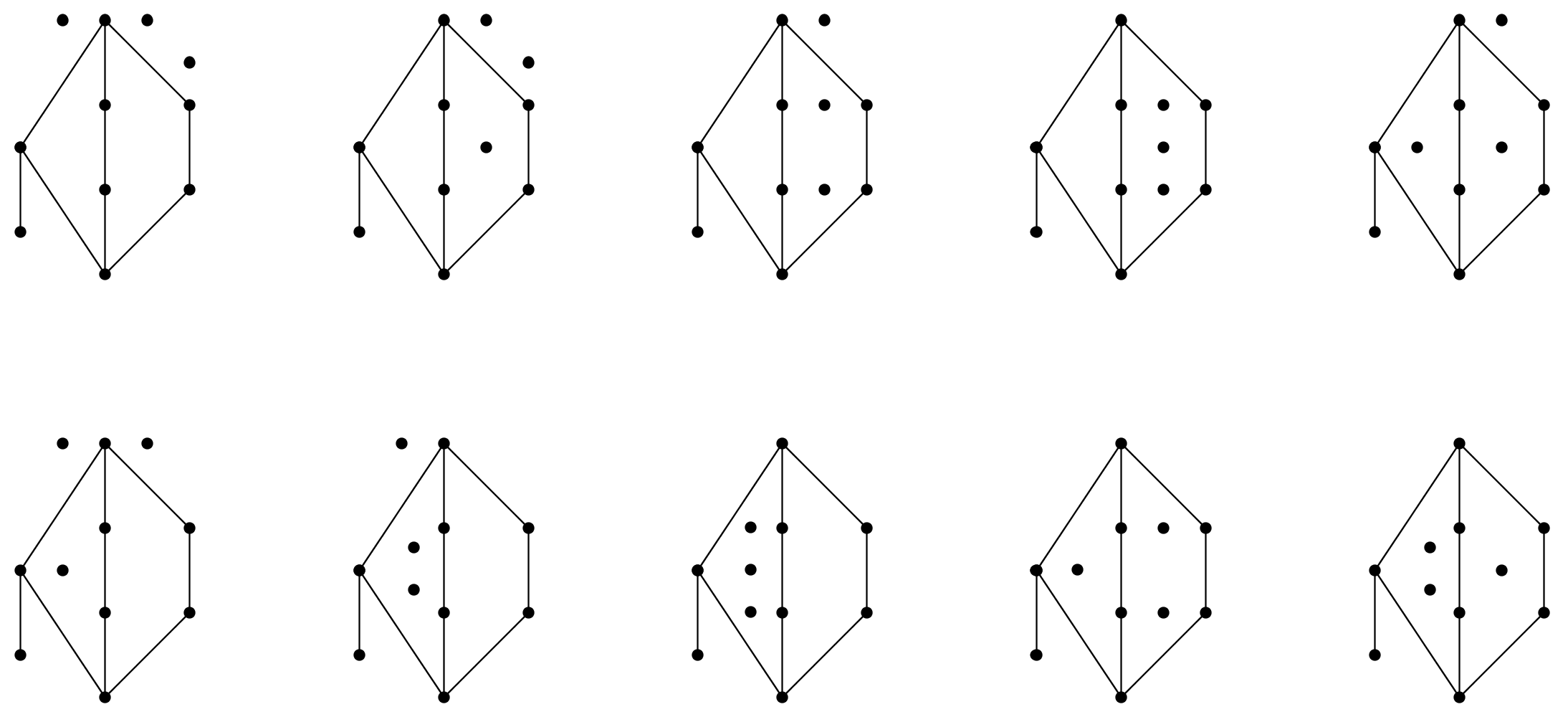}
\caption{\Deleas has ten \demb.}
\label{fig:D11a}
\end{figure}

\begin{figure}[H]
\centering
\includegraphics[width=0.9\linewidth, height=0.2\textheight]{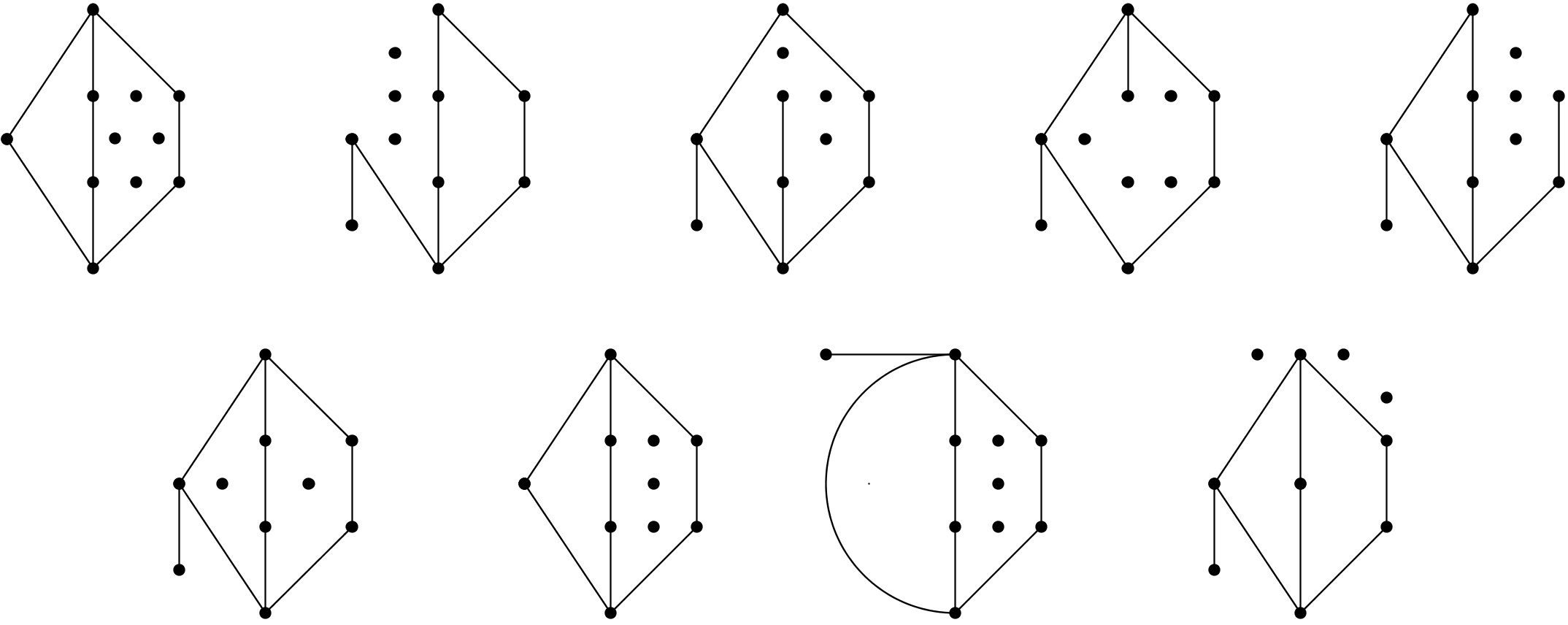}
\caption{\Deleas has four \equi; those that are incident to a vertex of degree 1, those that are incident to two vertices of degree 2, those that are incident to two vertices of degree 3, and those that are neither.}
\label{fig:D11aminors}
\end{figure}

\begin{figure}[H]
\centering
\includegraphics[width=0.6\linewidth, height=0.1\textheight]{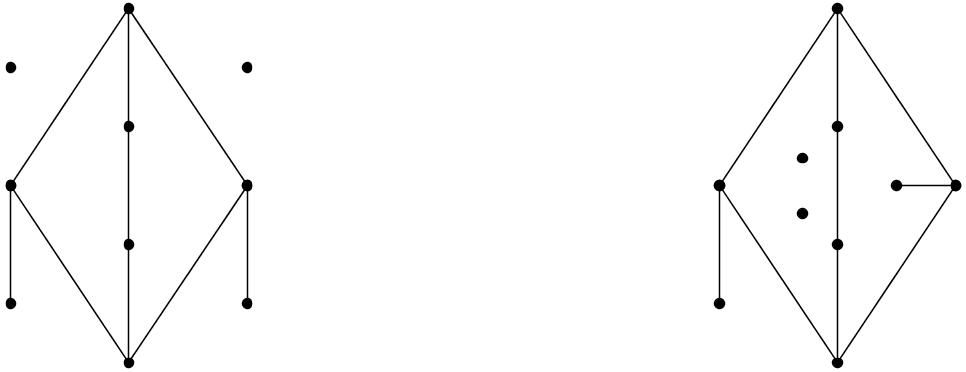}
\caption{\Dtwes has four \equi.}
\label{fig:D12}
\end{figure}

\begin{figure}[H]
\centering
\includegraphics[width=0.7\linewidth, height=0.18\textheight]{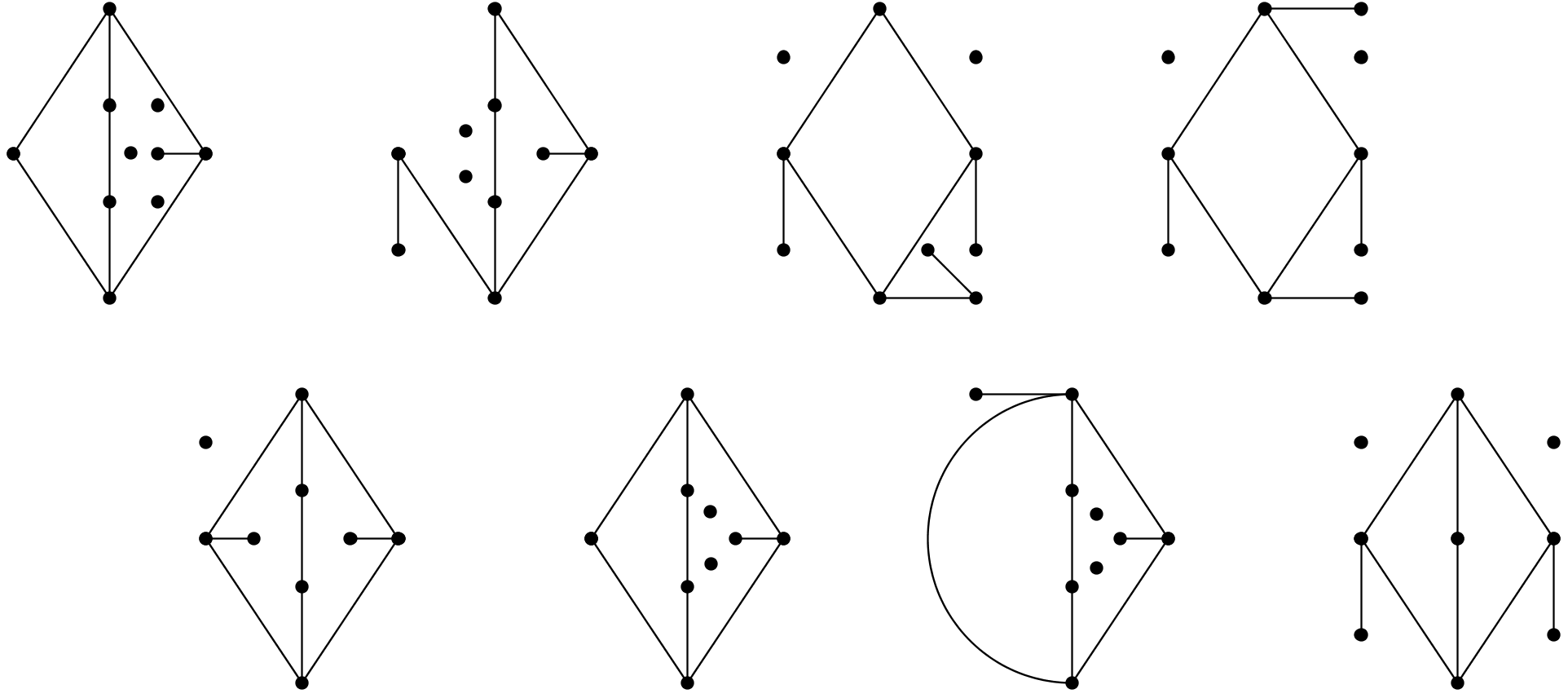}
\caption{\Dtwes has four \equi; those that are incident to a vertex of degree 1, those that are incident to two vertices of degree 2, those that are incident to two vertices of degree 3, and those that are neither.}
\label{fig:D12minors}
\end{figure}

\begin{figure}[H]
\centering
\vskip -.2in
\includegraphics[width=1\linewidth, height=0.2\textheight]{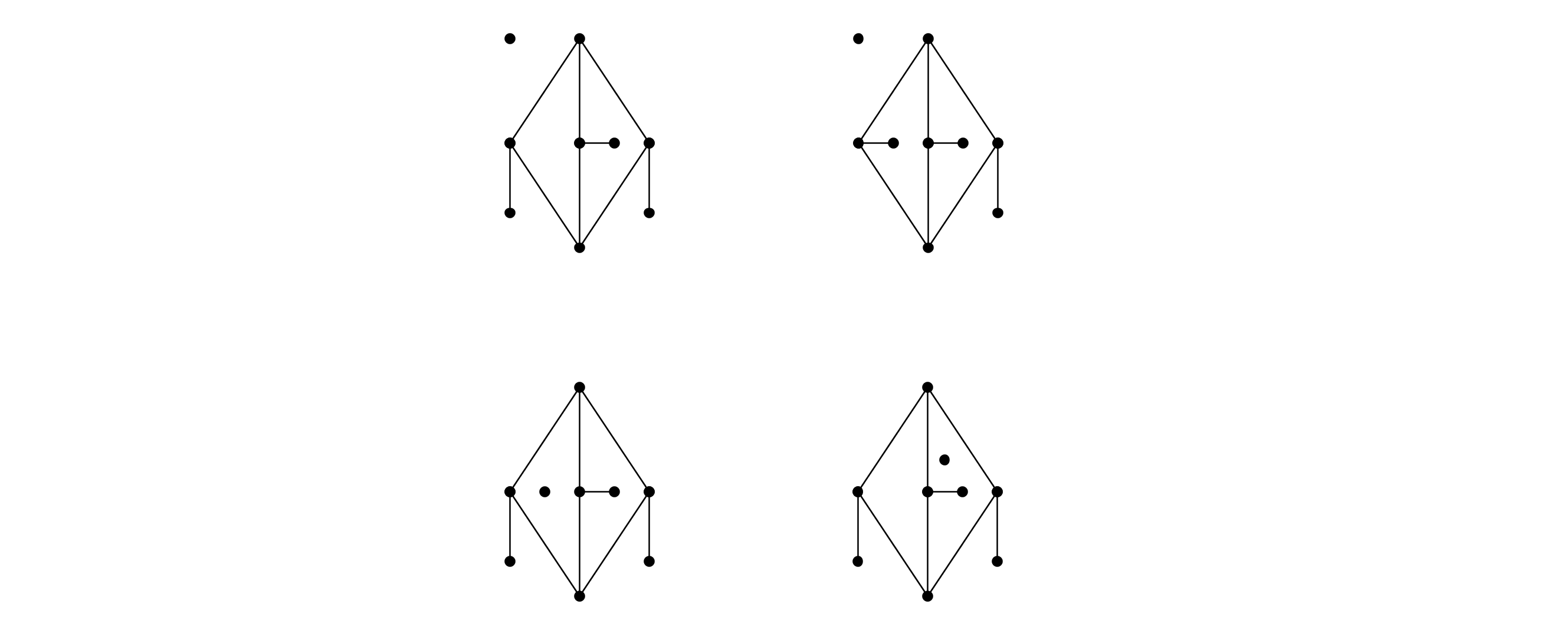}
\caption{\Dthis has four \demb.}
\label{fig:D13}
\end{figure}

\begin{figure}[H]
\centering
\includegraphics[width=0.7\linewidth, height=0.2\textheight]{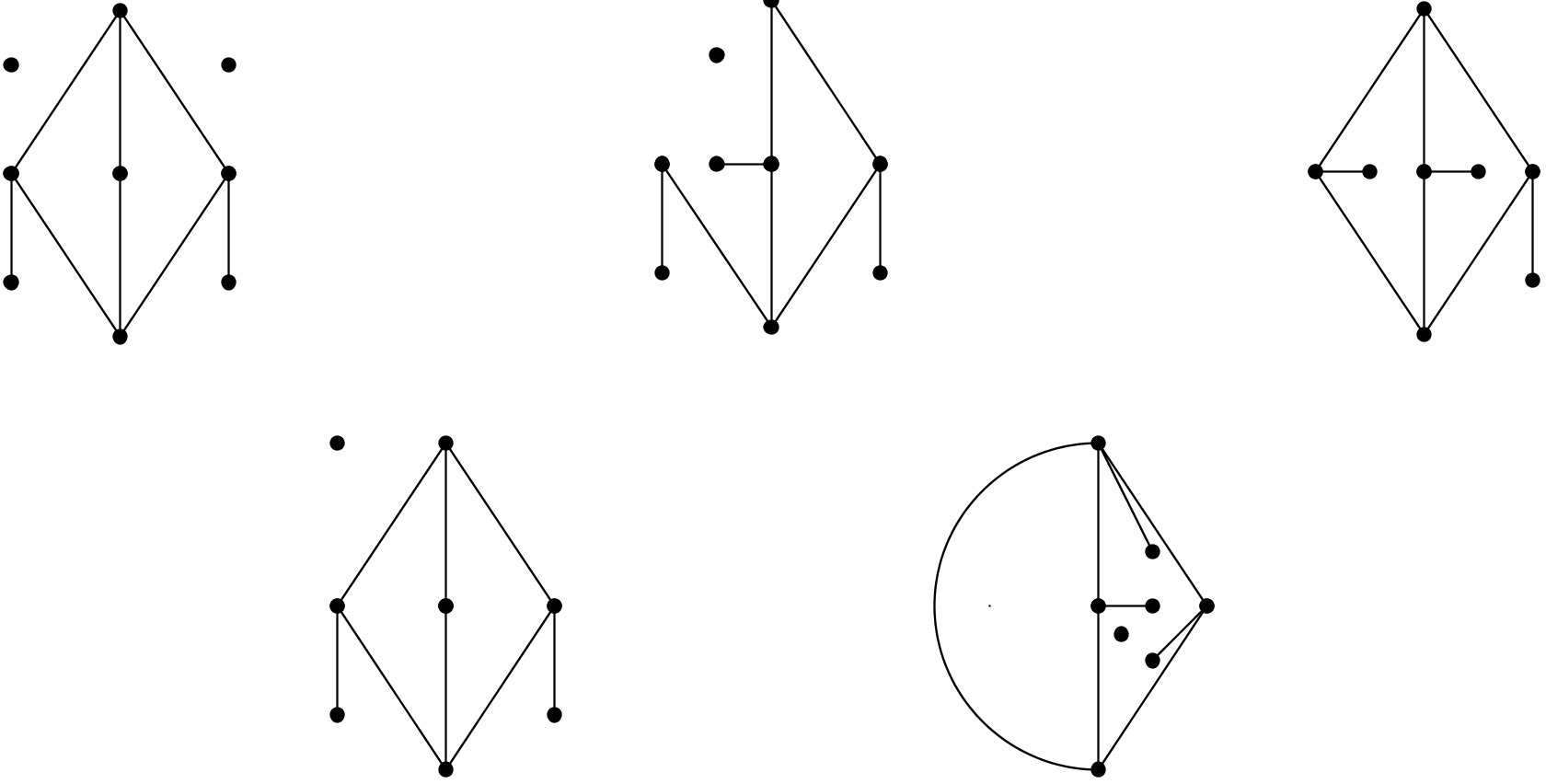}
\caption{\Dthis has only two \equi; those that are incident to a vertex of degree 1, and those that are not.}
\label{fig:D13minors}
\end{figure}


\begin{figure}[H]
\centering
\includegraphics[width=0.15\linewidth, height=0.1\textheight]{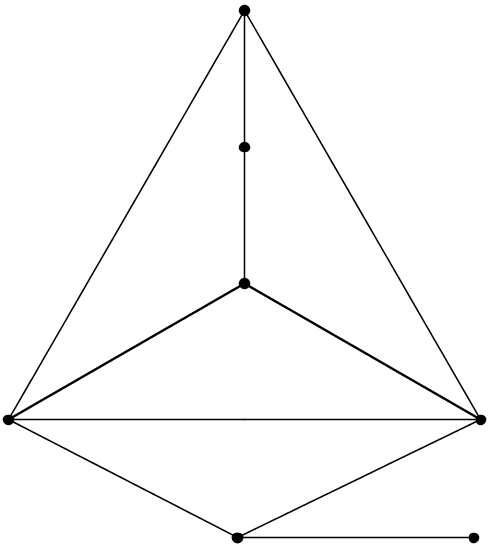}
\caption{\Dth$'$ \gosd \Dth.}
\label{fig:D3'}
\end{figure}

\begin{figure}[H]
\centering
\includegraphics[width=0.4\linewidth, height=0.1\textheight]{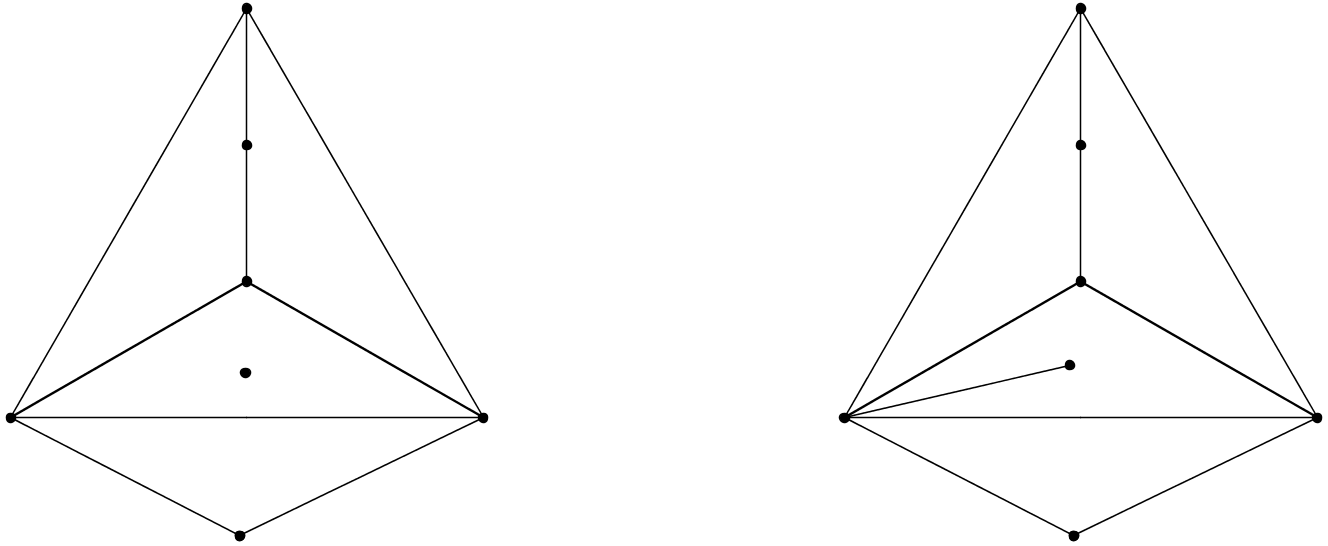}
\caption{\Dths \srsd.}
\label{fig:D3subdangle}
\end{figure}

\begin{figure}[H]
\centering
\includegraphics[width=0.2\linewidth, height=0.1\textheight]{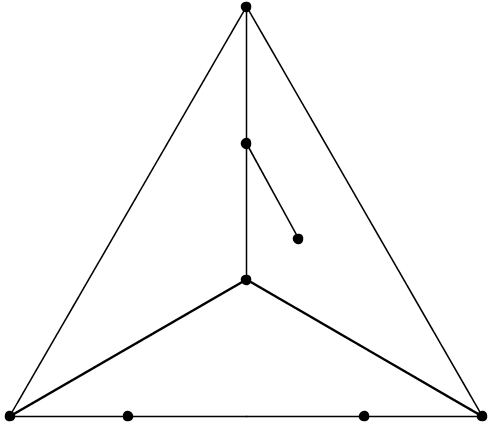}
\caption{\Dfo$'$ \gosd \Dfo.}
\label{fig:D4'}
\end{figure}

\begin{figure}[H]
\centering
\includegraphics[width=0.6\linewidth, height=0.1\textheight]{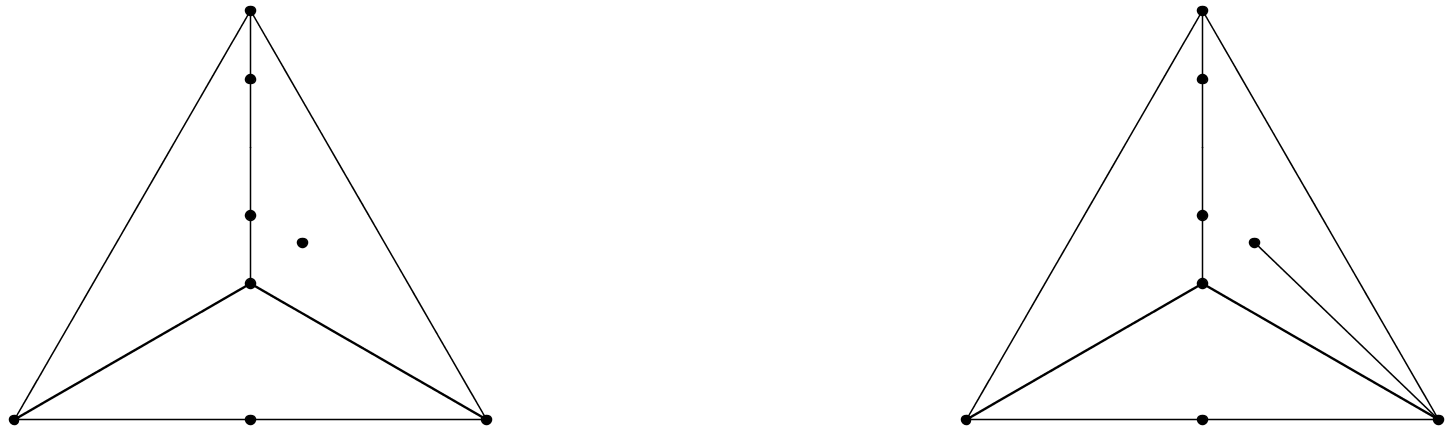}
\caption{\Dfos \srsd.}
\label{fig:D4subdangle}
\end{figure}

\begin{figure}[H]
\centering
\includegraphics[width=0.2\linewidth, height=0.1\textheight]{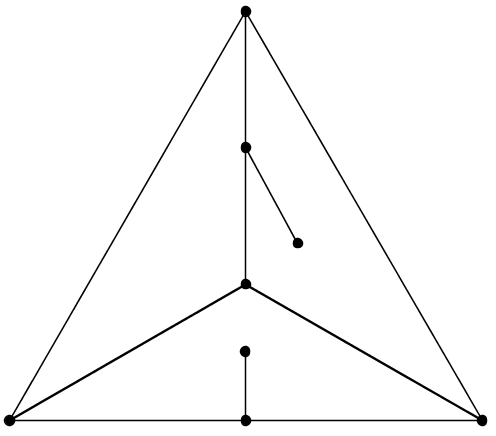}
\caption{\Dfo$''$ \gosd \Dfo$'$}
\label{fig:D4''}
\end{figure}

\begin{figure}[H]
\centering
\includegraphics[width=0.6\linewidth, height=0.1\textheight]{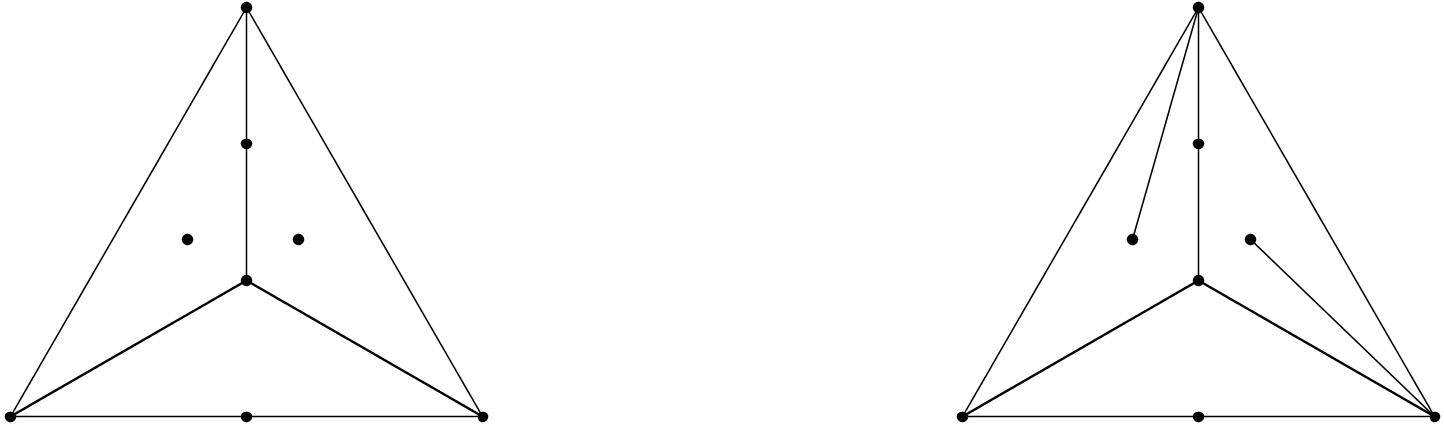}
\caption{\Dfo$'$ \srsd.}
\label{fig:D4'subdangle}
\end{figure}

\begin{figure}[H]
\centering
\includegraphics[width=0.3\linewidth, height=0.1\textheight]{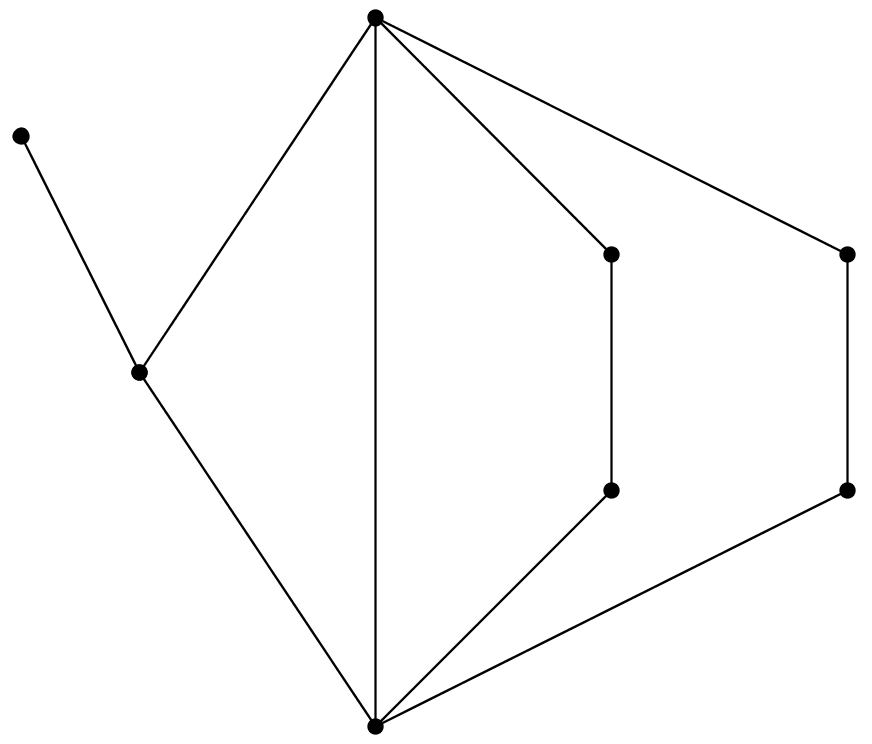}
\caption{\Dfi$'$ \gosd \Dfi.}
\label{fig:D5'}
\end{figure}

\begin{figure}[H]
\centering
\includegraphics[width=0.6\linewidth, height=0.1\textheight]{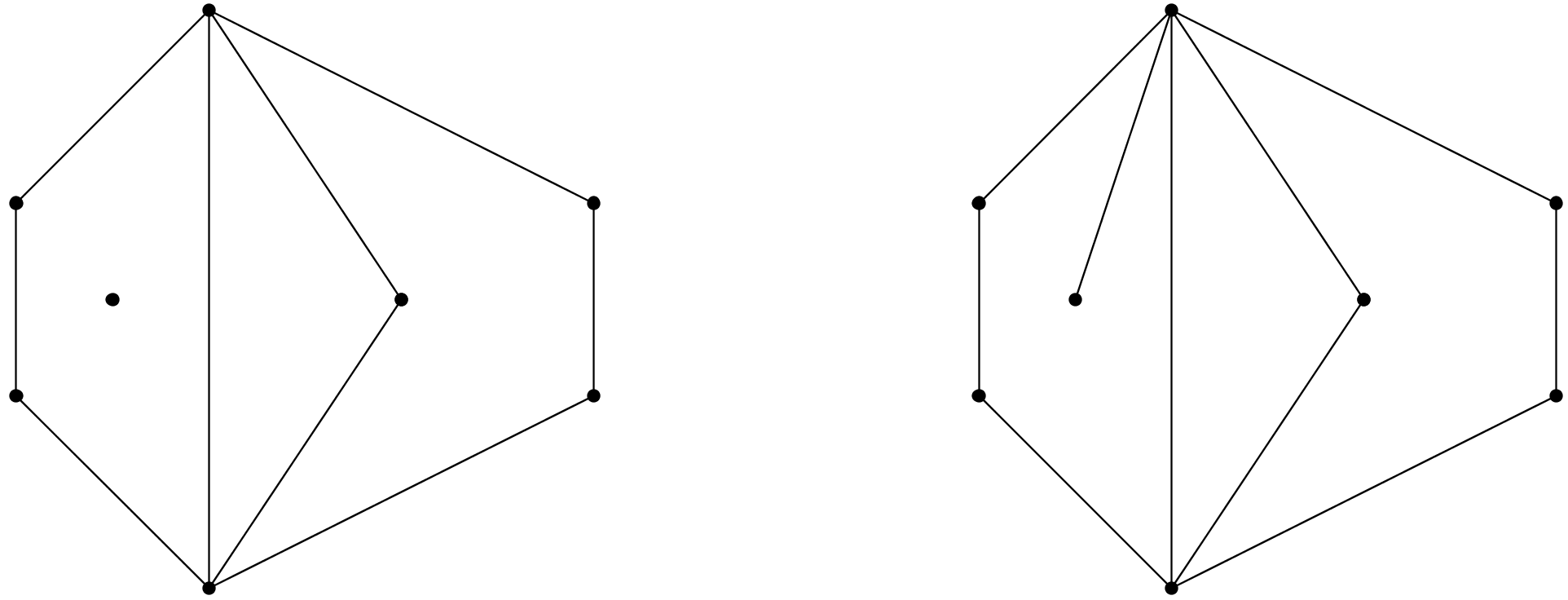}
\caption{\Dfis \srsd.}
\label{fig:D5subdangle}
\end{figure}

\begin{figure}[H]
\centering
\includegraphics[width=0.3\linewidth, height=0.1\textheight]{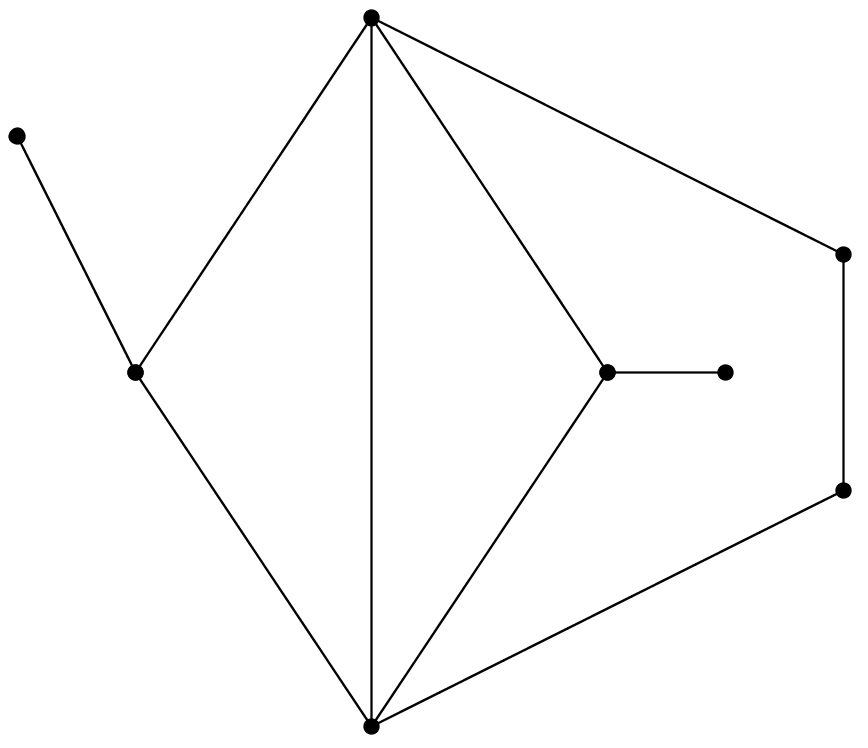}
\caption{\Dfi$''$ \gosd \Dfi$'$.}
\label{fig:D5''}
\end{figure}

\begin{figure}[H]
\centering
\includegraphics[width=0.6\linewidth, height=0.1\textheight]{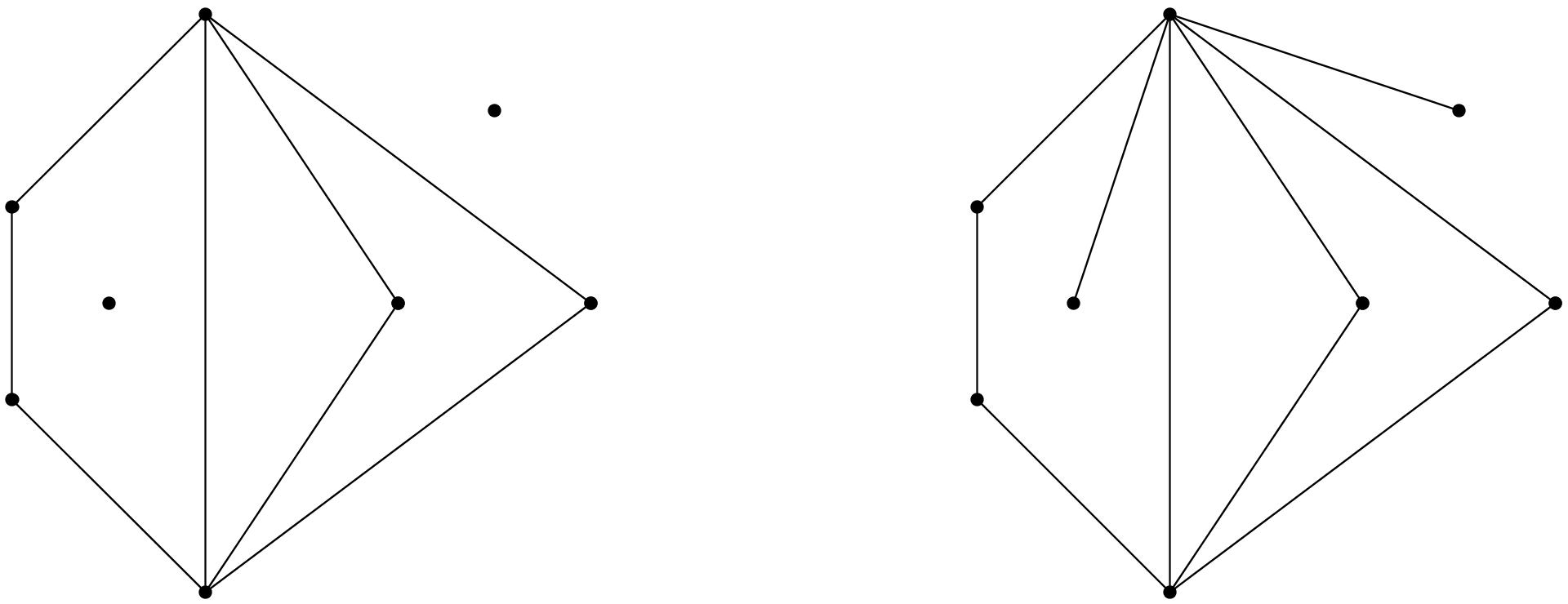}
\caption{\Dfi$'$ \srsd.}
\label{fig:D5'subdangle}
\end{figure}

\begin{figure}[H]
\centering
\includegraphics[width=0.3\linewidth, height=0.1\textheight]{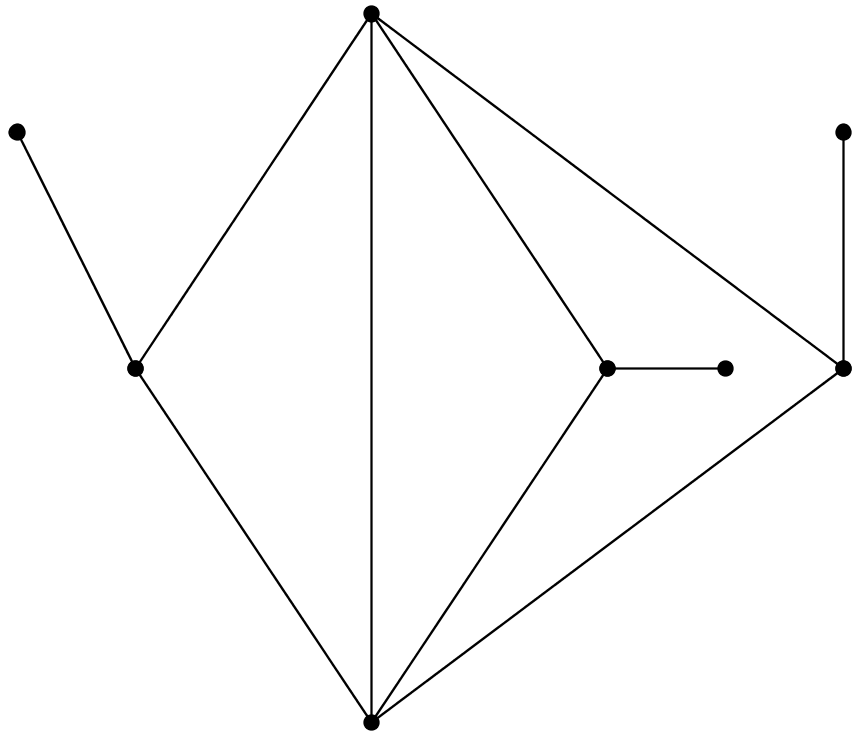}
\caption{\Dfi$'''$ \gosd \Dfi$''$.}
\label{fig:D5'''}
\end{figure}

\begin{figure}[H]
\centering
\includegraphics[width=0.6\linewidth, height=0.1\textheight]{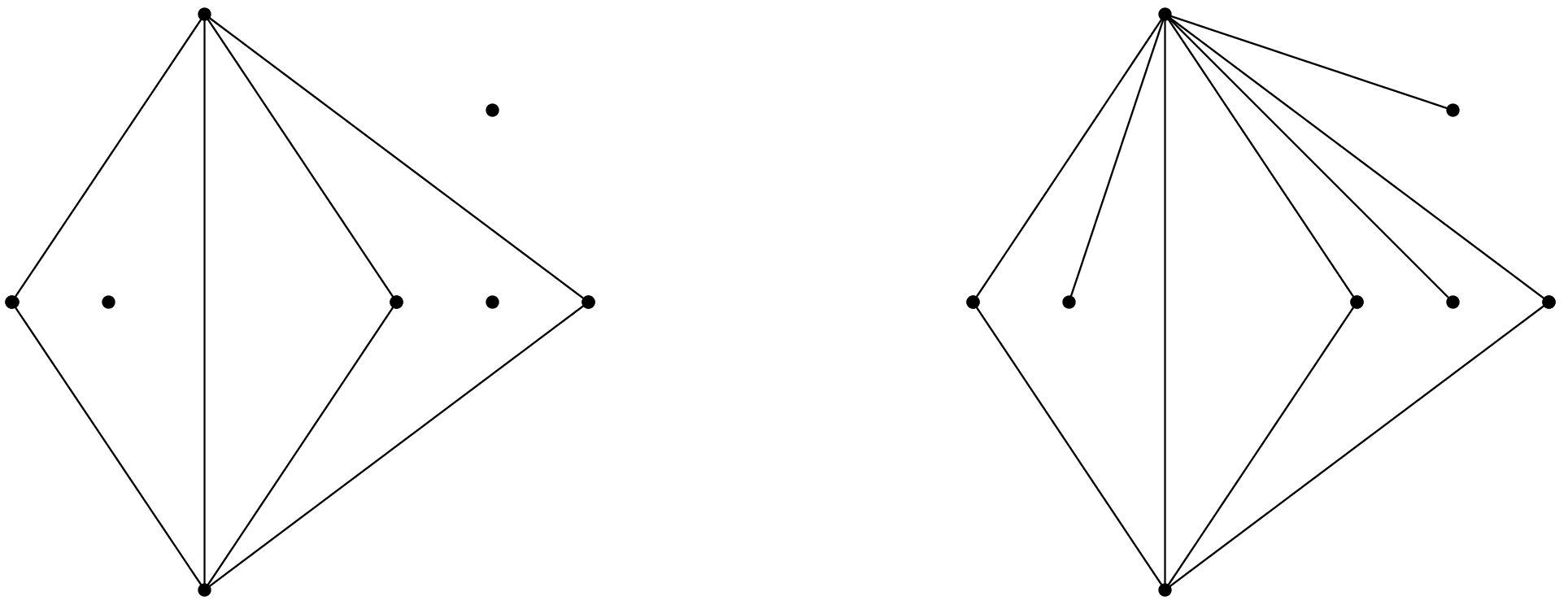}
\caption{\Dfi$''$ \srsd.}
\label{fig:D5''subdangle}
\end{figure}

\begin{figure}[H]
\centering
\includegraphics[width=0.25\linewidth, height=0.1\textheight]{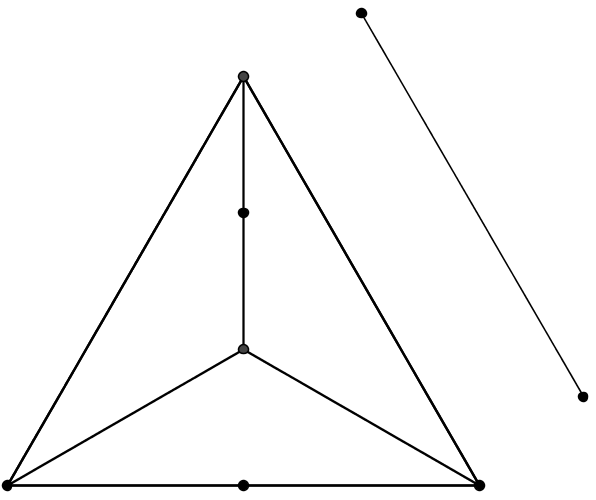}
\caption{\Dsebs \govb \Dsea.}
\label{fig:D7b}
\end{figure}

\begin{figure}[H]
\centering
\includegraphics[width=0.25\linewidth, height=0.1\textheight]{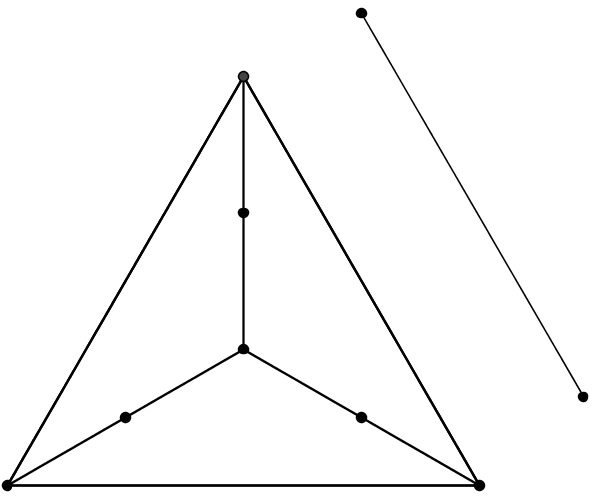}
\caption{\Deibs \govb \Deia.}
\label{fig:D8b}
\end{figure}

\begin{figure}[H]
\centering
\includegraphics[width=0.15\linewidth, height=0.1\textheight]{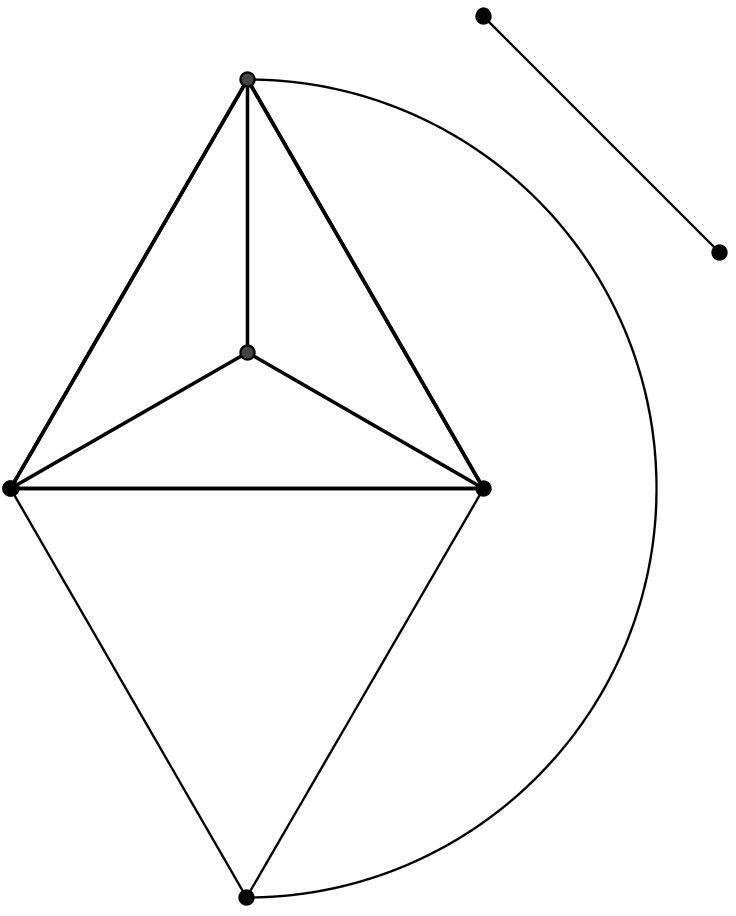}
\caption{\Dnibs \govb \Dnia.}
\label{fig:D9b}
\end{figure}

\begin{figure}[H]
\centering
\includegraphics[width=0.15\linewidth, height=0.1\textheight]{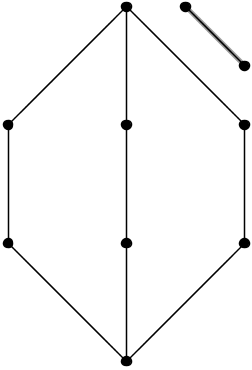}
\caption{\Dtenbs \govb \Dtena.}
\label{fig:D10b}
\end{figure}

\begin{figure}[H]
\centering
\includegraphics[width=0.15\linewidth, height=0.1\textheight]{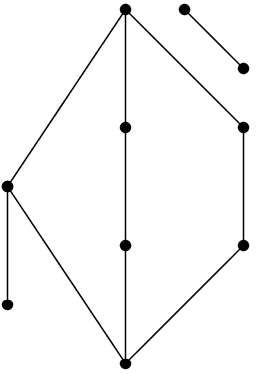}
\caption{\Delebs \govb \Delea.}
\label{fig:D11b}
\end{figure}

\section*{Acknowledgements} This research was conducted through the SUNY Potsdam/Clarkson University REU, with funding from the National Science Foundation under Grant No. DMA-1262737 and the National Security Administration under Grant No. H98230-14-1-0141.


\end{document}